\documentclass[8pt]{amsart}
\usepackage{amsmath}
\usepackage{amsthm}
\usepackage{amsbsy}
\usepackage{amssymb}
\usepackage{amsfonts}
\usepackage[dvips]{lscape}
\usepackage{xcolor}
\usepackage{amscd}
\usepackage[all,cmtip]{xy}
\usepackage{euscript}
\usepackage{parskip}
\usepackage{enumerate}
\usepackage[all,cmtip]{xy}
\usepackage{tikz-cd}
\usepackage{graphicx}
\usepackage{fancyhdr}
\usepackage{tikz-cd}
\usepackage{slashbox}
\usepackage{colortbl}
\usepackage{amssymb,amsmath,amsfonts,amsthm,epsfig,amscd,stmaryrd}
\usepackage{stmaryrd}

\usepackage[margin=0.5in]{geometry}
\usepackage[expansion=false]{microtype}

\usepackage[pdfusetitle,unicode,hidelinks]{hyperref}

\theoremstyle{plain}

\theoremstyle{plain}
\newtheorem{theorem}{Theorem}[section]

\newtheorem{proposition}[theorem]{Proposition}
\newtheorem{lemma}[theorem]{Lemma}
\newtheorem{corollary}[theorem]{Corollary}

\newtheorem{conjecture}[theorem]{Conjecture}

\theoremstyle{remark}

\newtheorem{question}[theorem]{Question}

\theoremstyle{definition}

\newcommand{\cal}{\EuScript}

\let\lim=\relax

\DeclareMathOperator*{\lim}{lim}

\renewcommand{\contentsname}{}

\begin{document}
\title{The union-closed sets conjecture for non-uniform distributions}
\author{Masoud Zargar}\footnotetext{\textit{Acknowledgment}: This project was supported by the University of Southern California. I thank Steven Heilman for his comments on a previous version of this paper.}
\address{University of Southern California, Los Angeles, California, United States}
\email{mzargar@usc.edu}
\renewcommand{\contentsname}{}
\maketitle
\vspace{-0.5cm}
\begin{abstract}
The union-closed sets conjecture, attributed to P\'eter Frankl from 1979, states that for any non-empty finite union-closed family of finite sets not consisting of only the empty set, there is an element that is in at least half of the sets in the family. We prove a version of Frankl's conjecture for families distributed according to any one of infinitely many distributions. As a corollary, in the intersection-closed reformulation of Frankl's conjecture, we obtain that it is true for families distributed according to any one of infinitely many Maxwell--Boltzmann distributions with inverse temperatures bounded below by a positive universal constant. Frankl's original conjecture corresponds to zero inverse temperature.
\end{abstract}
\setcounter{tocdepth}{1}
\tableofcontents
\section{Introduction}\label{intro}
The union-closed sets conjecture of Frankl from 1979 has attracted much interest due to its simplicity and difficulty. It states the following.
\begin{conjecture}\label{ucconj}(Frankl, 1979) For any non-empty finite family of finite sets $\cal{F}\neq\{\emptyset\}$ that is union-closed, that is, for any $A,B\in\cal{F}$ we have $A\cup B\in\cal{F}$, there is an element $i\in\bigcup_{A\in\cal{F}}A$ that is in at least half of the sets in $\cal{F}$.
\end{conjecture}
Recently, Gilmer~\cite{Gilmer} used information-theoretic methods to prove the first constant lower bound: there is an element that is in at least $0.01|\cal{F}|$ of the sets. He conjectured that similar methods should lead to the existence of an element that is in at least $\frac{3-\sqrt{5}}{2}|\cal{F}|\approx 0.38197|\cal{F}|$ of the sets. Within a few days, his method was refined by Alweiss--Huang--Selke~\cite{AHS}, Chase--Lovett~\cite{ChaseLovett}, Pebody~\cite{Pebody}, and Sawin~\cite{Sawin}, establishing this improvement. Sawin suggested in~\cite{Sawin} an idea for a further improvement to $\frac{3-\sqrt{5}}{2}$. This was independently investigated by Cambie~\cite{Cambie} and Yu~\cite{Yu}, leading to a constant $\approx 0.38234$.\\
\\
The idea of Gilmer was the following. Suppose $\cal{F}\subseteq 2^{[n]}$. We may then view each $A\in\cal{F}$ as a binary $n$-tuple $(X_1,\hdots,X_n)\in\{0,1\}^n$, thought of as a random variable on $\cal{F}$ with the uniform distribution. If $(Y_1,\hdots,Y_n)$ is an independent copy of the random variable $(X_1,\hdots,X_n)$, we have an inequality of Shannon entropies:
\[H(X_1+Y_1-X_1Y_1,\hdots,X_n+Y_n-X_nY_n)\leq H(X_1,\hdots,X_n).\]
This follows from the fact that $\cal{F}$ being union-closed implies that $(X_1+Y_1-X_1Y_1,\hdots,X_n+Y_n-X_nY_n)$ is a distribution supported on $\cal{F}$, and that the uniform distribution on $\cal{F}$ maximizes entropy. One then uses the chain rule for Shannon entropy to deduce that if no element is in at least $\frac{3-\sqrt{5}}{2}|\cal{F}|$ of the sets, then one obtains the contradiction that
\[H(X_1+Y_1-X_1Y_1,\hdots,X_n+Y_n-X_nY_n)>H(X_1,\hdots,X_n).\]
In his paper, Gilmer conjectured a statement involving KL-divergences from which Frankl's conjecture would have followed. However, Sawin~\cite{Sawin} and Ellis~\cite{Ellis} independently disproved his conjecture. It was also shown by Chase--Lovett~\cite{ChaseLovett} that for families that are \textit{approximately} union-closed (a relaxation of being union-closed), $\frac{3-\sqrt{5}}{2}$ is optimal.\\
\\
These information-theoretic results were preceded by results of Karpas~\cite{Karpas}, Balla--Bollob\'as--Eccles~\cite{BBE}, and many others. Balla--Bollob\'as--Eccles proved that Frankl's conjecture is true for union-closed families $\cal{F}\subseteq 2^{[n]}$ if $|\cal{F}|\geq \frac{2}{3}2^n$, while Karpas strengthened this to requiring $|\cal{F}|\geq 2^{n-1}$. Balister--Bollob\'as~\cite{BB} showed  that Frankl's conjecture is true with high probability for randomly chosen union-closed families.\\
\\
Frankl's conjecture is when the family $\cal{F}$ is given the uniform distribution. In this paper, we consider distributions on union-closed families of the following type. Given $\underline{k}:=(k_1,\hdots,k_n),\ \underline{m}:=(m_1,\hdots,m_n)\in\mathbb{N}^n$, and any given non-empty union-closed family of finite sets $\{\emptyset\}\neq\cal{F}\subseteq 2^{[n]}$, assign to each $X\in\cal{F}$ the probability
\[\frac{\prod_{j\in X}\frac{k_j}{m_j}}{\sum_{A\in\cal{F}}\prod_{j\in A}\frac{k_j}{m_j}}.\]
When $\underline{k}=\underline{m}$, this is the uniform distribution on $\cal{F}$. We prove the following version of Frankl's conjecture for an infinite family of such distributions.
\begin{theorem}\label{mainthmA}Suppose $\underline{k}:=(k_1,\hdots,k_n),\ \underline{m}:=(m_1,\hdots,m_n)\in\mathbb{N}^n$ such that for every $i$, $k_i\geq 5$ and $1\leq m_i\leq\sqrt{k_i}$. Then for every non-empty finite union-closed family $\cal{F}$ such that $\{\emptyset\}\neq\cal{F}\subseteq 2^{[n]}$ , there is an $i\in\cup_{A\in\cal{F}}A$ such that
\[\frac{\sum_{A\in\cal{F}:i\in A}\prod_{j\in A}\frac{k_j}{m_j}}{\sum_{A\in\cal{F}}\prod_{j\in A}\frac{k_j}{m_j}}\geq\frac{1}{2}.\]
\end{theorem}
Clearly, Frankl's conjecture is equivalent to a dual version for intersection-closed families of finite sets. As we will see, it will be convenient for us in the proof to consider this dual version. In this dual setting, Theorem~\ref{mainthmA} follows from the following theorem. 
\begin{theorem}\label{mainthm}Suppose $\underline{k}:=(k_1,\hdots,k_n),\ \underline{m}:=(m_1,\hdots,m_n)\in\mathbb{N}^n$ such that for every $i$, $k_i\geq 5$ and $1\leq m_i\leq\sqrt{k_i}$. Then for every finite intersection-closed family $\cal{F}\subseteq 2^{[n]}$ with $|\cal{F}|\geq 2$, there is an $i\in\cup_{A\in\cal{F}}A$ such that
\[f_i(\underline{k},\underline{m}):=\frac{\sum_{A\in\cal{F}:i\notin A}\prod_{j\in A}\frac{m_j}{k_j}}{\sum_{A\in\cal{F}}\prod_{j\in A}\frac{m_j}{k_j}}\geq\frac{1}{2}.\]
\end{theorem}
Taking $k_1=\hdots=k_n=k$ and $m_1=\hdots=m_n=m$, we obtain the following corollary.
\begin{corollary}\label{maincor}Let $k\geq 5$ and $1\leq m\leq \sqrt{k}$ be integers, and suppose $\cal{F}$ is a finite intersection-closed family of finite sets with $|\cal{F}|\geq 2$. Then there is an $i\in\cup_{A\in\cal{F}}A$ such that
\begin{equation}\label{mainineqthm}
f_i\left(\frac{m}{k}\right):=\frac{\sum_{A\in\cal{F}:i\notin A}\left(\frac{m}{k}\right)^{|A|}}{\sum_{A\in\cal{F}}\left(\frac{m}{k}\right)^{|A|}}\geq\frac{1}{2}.
\end{equation}
\end{corollary}
In particular, from Corollary~\ref{maincor}, we have that given integers $k\geq 5$ and $1\leq m\leq\sqrt{k}$, for any nonempty union-closed family $\cal{F}\neq\{\emptyset\}$ of finite sets, there is an $i\in\cup_{A\in\cal{F}}A$ such that 
\[\frac{\sum_{A\in\cal{F}:i\in A}\left(\frac{k}{m}\right)^{|A|}}{\sum_{A\in\cal{F}}\left(\frac{k}{m}\right)^{|A|}}\geq\frac{1}{2}.\]
\vspace{3mm}
\begin{center}
\textit{In the rest of this paper, our family $\cal{F}$ is a finite \underline{intersection}-closed family of finite sets with $|\cal{F}|\geq 2$.}
\end{center}
\vspace{3mm}
Corollary~\ref{maincor} is related to Maxwell--Boltzmann distributions in the following way. For each $t\in (0,1]$, we may consider the probability distribution on $\cal{F}$ with probability
\[\frac{t^{|X|}}{\sum_{A\in\cal{F}}t^{|A|}}\]
assigned to each $X\in\cal{F}$. When $t=1$, we have the uniform distribution. We may ask about Frankl's conjecture with respect to these distributions for different $t$. Note that since
\[\frac{t^{|X|}}{\sum_{A\in\cal{F}}t^{|A|}}=\frac{e^{-|X|\log \left(\frac{1}{t}\right)}}{\sum_{A\in\cal{F}}e^{-|A|\log \left(\frac{1}{t}\right)}},\]
different values of $t$ give different Maxwell--Boltzmann distribution on $\cal{F}$ with inverse temperature $\beta=\log\left(\frac{1}{t}\right)\geq 0$. From the perspective of statistical mechanics, this gives us the distribution on $\cal{F}$ obtained by assigning the energies $|X|$ to $X\in\cal{F}$, and having maximal entropy conditional on fixed expected energy
\[\frac{\sum_{A\in\cal{F}}|A|e^{-|A|\log \left(\frac{1}{t}\right)}}{\sum_{A\in\cal{F}}e^{-|A|\log \left(\frac{1}{t}\right)}}=\frac{\sum_{A\in\cal{F}}|A|t^{|A|}}{\sum_{A\in\cal{F}}t^{|A|}}.\]
For each $i\in\cup_{A\in\cal{F}}A$, define
\[f_i(t):=\frac{\sum_{A\in\cal{F}:i\notin A}t^{|A|}}{\sum_{A\in\cal{F}}t^{|A|}}.\]
Conjecture~\ref{ucconj} is equivalent to the statement that for any finite intersection-closed family $\cal{F}$ of finite sets with $|\cal{F}|\geq 2$, there is an element $i\in\cup_{A\in\cal{F}}A$ such that
\[f_i(1)\geq\frac{1}{2}.\]
Therefore, Frankl's conjecture corresponds to vanishing inverse temperature $\beta$. Corollary~\ref{maincor} is for low enough temperatures ($\beta=\log\left(\frac{k}{m}\right)\geq\log \left(\frac{k}{\lfloor\sqrt{k}\rfloor}\right)\geq\log\left(\frac{5}{2}\right)\approx 0.91629$), while Frankl's conjecture is the high temperature ($\beta=0$) version of Corollary~\ref{maincor}. Note that even Corollary~\ref{maincor} is neither weaker nor stronger than Frankl's conjecture. Note that $k$ and $m$ are chosen independently of $\cal{F}$, in particular, independently of $n$.\\
\\
Our method applies to other values of $k,m$ with $k>2$; however, we do not obtain $\frac{1}{2}$ in the inequality of Theorem~\ref{mainthm}. For example, for $k=4$ and $m=2$, we obtain that for some $i\in\cup_{A\in\cal{F}}A$,
\begin{equation}\label{km42}
f_i\left(\frac{1}{2}\right):=\frac{\sum_{A\in\cal{F}:i\notin A}\left(\frac{1}{2}\right)^{|A|}}{\sum_{A\in\cal{F}}\left(\frac{1}{2}\right)^{|A|}}\geq 0.469.
\end{equation}
The right hand side of~\eqref{km42} could be $\frac{1}{2}$ if we could take $k=2, m=1$. However, showing the concavity of $F_{2,1}:\cal{M}_{\phi}\rightarrow\mathbb{R}$, if true, in Proposition~\ref{concavity} requires more work than when $k\neq 2$. For $k=5$ and $m=3$ (which does not satisfy $1\leq m\leq \sqrt{k}$), we obtain that for some $i\in\cup_{A\in\cal{F}}A$,
\begin{equation}\label{km53}
f_i\left(\frac{3}{5}\right):=\frac{\sum_{A\in\cal{F}:i\notin A}\left(\frac{3}{5}\right)^{|A|}}{\sum_{A\in\cal{F}}\left(\frac{3}{5}\right)^{|A|}}\geq 0.385.
\end{equation}
\begin{question}For which other distributions not included in the above do we have Frankl's conjecture? For example, is Frankl's conjecture true for intersection-closed families distributed according to Maxwell--Boltzmann distributions with inverse temperatures not addressed in Corollary~\ref{maincor}?
\end{question}
We now sketch our proof of Theorem~\ref{mainthm}. For every pair of $n$-tuples of natural numbers $\underline{k}:=(k_1,\hdots,k_n),\ \underline{m}:=(m_1,\hdots,m_n)\in\mathbb{N}^n$, we consider the commutative ring
\[R_{k,m}:=\mathbb{Z}[\varepsilon_1,\hdots,\varepsilon_n,\zeta_1,\hdots,\zeta_n]/\left<\varepsilon_1^{k_1},\hdots,\varepsilon_n^{k_n},\zeta_1^{m_1}-1,\hdots,\zeta_n^{m_n}-1,\varepsilon_1(\zeta_1-1),\hdots,\varepsilon_n(\zeta_n-1)\right>.\]
The full ring structure is not important. The importance is that each $\zeta_i$ is thought of as primitive $m_i$-th root of unity, $\varepsilon_i$ is a nilpotent element with $k_i$ the smallest natural number such that $\varepsilon_i^{k_i}=0$, and $\varepsilon_i^j\zeta_i^{\ell}=\varepsilon_i^j$ for every $j\geq 1$ and $\ell\geq 0$. We let $\varepsilon_i^0=\zeta_i^0=1$. Let
\[\Theta:R_{k,m}^n\rightarrow\mathbb{Z}^n\]
be the product of $n$ copies of the ring homomorphism $R_{k,m}\rightarrow\mathbb{Z}$ given by sending each $\varepsilon_i$ to $0$ and each $\zeta_i$ to $1$. For each intersection-closed family $\cal{F}$ of subsets of $2^{[n]}$ with $|\cal{F}|\geq 2$, naturally viewed as a set of binary $n$-tuples closed under component-wise multiplication, we consider the set of $n$-tuples
\[\boldsymbol{\cal{F}}_{\underline{k},\underline{m}}:=\Big\{(\boldsymbol{X}^{\underline{k},\underline{m}}_1,\hdots,\boldsymbol{X}^{\underline{k},\underline{m}}_n)\in\prod_{\ell=1}^n\{0,1,\varepsilon_{\ell},\hdots,\varepsilon_{\ell}^{k_{\ell}-1},\zeta_{\ell},\hdots,\zeta_{\ell}^{m_{\ell}-1}\}:\Theta(\boldsymbol{X}^{\underline{k},\underline{m}}_1,\hdots,\boldsymbol{X}^{\underline{k},\underline{m}}_n)\in\cal{F}\Big\}.\]
In words, we are enlarging $\cal{F}$ so that for each $(X_1,\hdots,X_n)\in\cal{F}$, we are including all those $n$-tuples in $\prod_{\ell=1}^n\{0,1,\varepsilon_{\ell},\hdots,\varepsilon_{\ell}^{k_{\ell}-1},\zeta_{\ell},\hdots,\zeta_{\ell}^{m_{\ell}-1}\}$ obtained by replacing each $X_i$ with any positive power of $\varepsilon_i$ if $X_i=0$, and any power of $\zeta_i$ if $X_i=1$. Since $\cal{F}$ is closed under component-wise multiplication, $\boldsymbol{\cal{F}}_{\underline{k},\underline{m}}$ is also closed under component-wise multiplication. There is the natural random variable 
\[(\boldsymbol{X}^{\underline{k},\underline{m}}_1,\hdots,\boldsymbol{X}^{\underline{k},\underline{m}}_n):\boldsymbol{\cal{F}}_{\underline{k},\underline{m}}\hookrightarrow\prod_{\ell=1}^n\{0,1,\varepsilon_{\ell},\hdots,\varepsilon_{\ell}^{k_{\ell}-1},\zeta_{\ell},\hdots,\zeta_{\ell}^{m_{\ell}-1}\}\]
with the uniform distribution on $\boldsymbol{\cal{F}}_{\underline{k},\underline{m}}$. It will turn out that the uniform distribution on $\boldsymbol{\cal{F}}_{\underline{k},\underline{m}}$ is related to $\cal{F}$ equipped with the distribution given in Theorem~\ref{mainthm}.\\
\\
Since $\boldsymbol{\cal{F}}_{\underline{k},\underline{m}}$ is closed under component-wise multiplication, and the uniform distribution on $\boldsymbol{\cal{F}}_{\underline{k},\underline{m}}$ maximizes entropy, we have for any independent uniformly distributed copy $(\boldsymbol{Y}^{\underline{k},\underline{m}}_1,\hdots,\boldsymbol{Y}^{\underline{k},\underline{m}}_n)$ of $(\boldsymbol{X}^{\underline{k},\underline{m}}_1,\hdots,\boldsymbol{X}^{\underline{k},\underline{m}}_n)$ that

\begin{equation}\label{basicineq}
H\left(\boldsymbol{X}^{\underline{k},\underline{m}}_1\boldsymbol{Y}^{\underline{k},\underline{m}}_1,\hdots,\boldsymbol{X}^{\underline{k},\underline{m}}_n\boldsymbol{Y}^{\underline{k},\underline{m}}_n\right)\leq H(\boldsymbol{X}^{\underline{k},\underline{m}}_1,\hdots,\boldsymbol{X}^{\underline{k},\underline{m}}_n).
\end{equation}

By the chain rule, we have the following equalities involving conditional Shannon entropies:

\begin{equation}\label{chain1}
H(\boldsymbol{X}^{\underline{k},\underline{m}}_1,\hdots,\boldsymbol{X}^{\underline{k},\underline{m}}_n)=\sum_{i=1}^nH(\boldsymbol{X}^{\underline{k},\underline{m}}_i|\boldsymbol{X}^{\underline{k},\underline{m}}_{<i})
\end{equation}

and

\begin{eqnarray}
H\left(\boldsymbol{X}^{\underline{k},\underline{m}}_1\boldsymbol{Y}^{\underline{k},\underline{m}}_1,\hdots,\boldsymbol{X}^{\underline{k},\underline{m}}_n\boldsymbol{Y}^{\underline{k},\underline{m}}_n\right)&=&\sum_{i=1}^nH\left(\boldsymbol{X}^{\underline{k},\underline{m}}_i\boldsymbol{Y}^{\underline{k},\underline{m}}_i|\boldsymbol{X}^{\underline{k},\underline{m}}_1\boldsymbol{Y}^{\underline{k},\underline{m}}_1,\hdots,\boldsymbol{X}^{\underline{k},\underline{m}}_{i-1}\boldsymbol{Y}^{\underline{k},\underline{m}}_{i-1}\right)\\
\label{chain2}&\geq&\sum_{i=1}^nH\left(\boldsymbol{X}^{\underline{k},\underline{m}}_i\boldsymbol{Y}^{\underline{k},\underline{m}}_i|\boldsymbol{X}^{\underline{k},\underline{m}}_{<i},\boldsymbol{Y}^{\underline{k},\underline{m}}_{<i}\right),
\end{eqnarray}
where $\boldsymbol{X}^{\underline{k},\underline{m}}_{<i}:=(\boldsymbol{X}^{\underline{k},\underline{m}}_1,\hdots,\boldsymbol{X}^{\underline{k},\underline{m}}_{i-1})$. The last inequality follows from the information processing inequality for conditional Shannon entropies.\\
\\
Consider the space $\cal{M}_{\phi}$ of probability distributions on $[0,1]$ with fixed expectation $\mathop{\mathbb{E}}_{x\sim\mu}[x]=\phi$. $\cal{M}_{\phi}$ is a convex space that is compact with respect to the weak-$*$ topology. Given integers $k,m\geq 1$, consider the functions
\[h_{k,m}(t):=-kt\log t-(1-kt)\log\left(1-kt\right)+(1-kt)\log m\]
and
\begin{eqnarray*}g_{k,m}(x,y)&&:=-(1-kx)(1-ky)\log\left((1-kx)(1-ky)\right)-\left(x+y+\left(\frac{k(k-1)}{2}-1\right)xy\right)\log\left(x+y+\left(\frac{k(k-1)}{2}-1\right)xy\right)\\&&-\sum_{j=1}^{k-1}\left(x+y-(2k-j+1)xy\right)\log\left(x+y-(2k-j+1)xy\right)+(1-kx)(1-ky)\log m.
\end{eqnarray*}
Let $F_{k,m}:\cal{M}_{\phi}\rightarrow\mathbb{R}$ be given by
\[F_{k,m}(\mu):=\mathop{\mathbb{E}}_{(x,y)\sim\mu\times\mu}\Big[g_{k,m}\left(\frac{x}{k},\frac{y}{k}\right)\Big]-\mathop{\mathbb{E}}_{x\sim\mu}\Big[h_{k,m}\left(\frac{x}{k}\right)\Big].\]
When $k=m=1$, this essentially specializes to the functional that appeared in Gilmer~\cite{Gilmer} and subsequent papers by others, and whose optimization led to the $\frac{3-\sqrt{5}}{2}$ bound for Frankl's conjecture. $F_{k,m}$ is a continuous map. In Section~\ref{entropies}, we prove that for each $i$, we have
\[H\left(\boldsymbol{X}^{\underline{k},\underline{m}}_i\boldsymbol{Y}^{\underline{k},\underline{m}}_i|\boldsymbol{X}^{\underline{k},\underline{m}}_{<i},\boldsymbol{Y}^{\underline{k},\underline{m}}_{<i}\right)-H(\boldsymbol{X}^{\underline{k},\underline{m}}_i|\boldsymbol{X}^{\underline{k},\underline{m}}_{<i})=F_{k_i,m_i}(\mu_{\underline{k},\underline{m}}(i)),\]
where
\[\mu_{\underline{k},\underline{m}}(i):=\sum_{\boldsymbol{a}\in\prod_{\ell=1}^{i-1}\{0,1,\varepsilon_{\ell},\hdots,\varepsilon_{\ell}^{k_{\ell}-1},\zeta_{\ell},\hdots,\zeta_{\ell}^{m_{\ell}-1}\}}\mathbb{P}_{\boldsymbol{\cal{F}}_{\underline{k},\underline{m}}}\Big[\boldsymbol{X}^{\underline{k},\underline{m}}_{<i}=\boldsymbol{a}\Big]\delta_{k_i\mathbb{P}_{\boldsymbol{\cal{F}}_{\underline{k},\underline{m}}}\Big[\boldsymbol{X}^{\underline{k},\underline{m}}_i=0|\boldsymbol{X}^{\underline{k},\underline{m}}_{<i}=\boldsymbol{a}\Big]}.\]
Note that $\mathop{\mathbb{E}}_{x\sim\mu_{\underline{k},\underline{m}}(i)}=k_i\mathbb{P}_{\boldsymbol{\cal{F}}_{\underline{k},\underline{m}}}\Big[\boldsymbol{X}^{\underline{k},\underline{m}}_i=0\Big]$. Since $|\cal{F}|\geq 2$, there must be an $i$ such that $k_i\mathbb{P}_{\boldsymbol{\cal{F}}_{\underline{k},\underline{m}}}\Big[\boldsymbol{X}^{\underline{k},\underline{m}}_i=0\Big]>0$.
In this paper, we prove the following general theorem whose proof will be sketched at the end of this introduction.
\begin{theorem}\label{functionalthm}Suppose $k\geq 5$ and $1\leq m\leq\sqrt{k}$ are integers. Then for $0<\phi<\frac{1}{2}$, $F_{k,m}(\mu)>0$ for every $\mu\in\cal{M}_{\phi}$.
\end{theorem}
Theorem~\ref{functionalthm} has the following corollary.
\begin{corollary}\label{mainprop}Suppose $\underline{k},\underline{m}$ are as before such that for some $i$, $k_i\geq 5$ and $1\leq m_i\leq \sqrt{k_i}$. If 
\begin{equation}\label{cond}0<k_i\mathbb{P}_{\boldsymbol{\cal{F}}_{\underline{k},\underline{m}}}\Big[\boldsymbol{X}^{\underline{k},\underline{m}}_i=0\Big]<\frac{1}{2},
\end{equation} 
then
\[H\left(\boldsymbol{X}^{\underline{k},\underline{m}}_i\boldsymbol{Y}^{\underline{k},\underline{m}}_i|\boldsymbol{X}^{\underline{k},\underline{m}}_{<i},\boldsymbol{Y}^{\underline{k},\underline{m}}_{<i}\right)>H(\boldsymbol{X}^{\underline{k},\underline{m}}_i|\boldsymbol{X}^{\underline{k},\underline{m}}_{<i}).\]
\end{corollary}
Given $\underline{k}:=(k_1,\hdots,k_n),\ \underline{m}:=(m_1,\hdots,m_n)\in\mathbb{N}^n$ such that for \textit{every} $i$, $k_i\geq 5$ and $1\leq m_i\leq\sqrt{k_i}$, since we cannot have
\begin{equation}\label{basiccontradiction}
H\left(\boldsymbol{X}^{\underline{k},\underline{m}}_1\boldsymbol{Y}^{\underline{k},\underline{m}}_1,\hdots,\boldsymbol{X}^{\underline{k},\underline{m}}_n\boldsymbol{Y}^{\underline{k},\underline{m}}_n\right)>H(\boldsymbol{X}^{\underline{k},\underline{m}}_1,\hdots,\boldsymbol{X}^{\underline{k},\underline{m}}_n),
\end{equation}
Corollary~\ref{mainprop} along with~\eqref{chain1} and~\eqref{chain2} imply that there is an $i$ such that
\[k_i\mathbb{P}_{\boldsymbol{\cal{F}}_{\underline{k},\underline{m}}}\Big[\boldsymbol{X}^{\underline{k},\underline{m}}_i=0\Big]\geq\frac{1}{2}.\]
However,
\[k_i\mathbb{P}_{\boldsymbol{\cal{F}}_{\underline{k},\underline{m}}}\Big[\boldsymbol{X}^{\underline{k},\underline{m}}_i=0\Big]=k_i\frac{\sum_{A\in\cal{F}:i\notin A}\prod_{j\in A}m_j\prod_{i\neq\ell\notin A}k_{\ell}}{\sum_{A\in\cal{F}}\prod_{j\in A}m_j\prod_{\ell\notin A}k_{\ell}}=\frac{\sum_{A\in\cal{F}:i\notin A}\prod_{j\in A}\frac{m_j}{k_j}}{\sum_{A\in\cal{F}}\prod_{j\in A}\frac{m_j}{k_j}}=:f_i(\underline{k},\underline{m}),\]
from which Theorem~\ref{mainthm} follows.\\
\\
Theorem~\ref{functionalthm} will be proved as follows. In Proposition~\ref{concavity}, we prove that for integers $k\geq 3$ and any $m\geq 1$, the continuous map $F_{k,m}:\cal{M}_{\phi}\rightarrow\mathbb{R}$ is concave. Furthermore, $\cal{M}_{\phi}$ is a compact space with respect to the weak-$*$ topology, and so $F_{k,m}$ attains its minimizers in $\cal{M}_{\phi}$. In Lemma~\ref{measuretypes}, we classify the measures that are potential minimizers. When $\phi<\frac{1}{2}$, they turn out to be one of the following three types of finitely-supported probability measures in $\cal{M}_{\phi}$:
\begin{enumerate}[I)]
\item $\delta_{\phi}$;
\item $(1-p)\delta_0+p\delta_x$, where $0<p<1$ and $px=\phi$; and
\item $(1-p)\delta_1+p\delta_x$, where $0<p<1$ and $p=\frac{1-\phi}{1-x}$.
\end{enumerate}
We then bound from below $F_{k,m}$ for each of these three types of measures to obtain Theorem~\ref{functionalthm}.
\section{Difference of entropies}\label{entropies}
In order to prove Corollary~\ref{mainprop}, we begin by rewriting the conditional entropies using equalities~\eqref{zero} and~\eqref{one} to obtain the following proposition. 
\begin{proposition}\label{differenceprop}For each $i$, we have
\[H\left(\boldsymbol{X}^{\underline{k},\underline{m}}_i\boldsymbol{Y}^{\underline{k},\underline{m}}_i|\boldsymbol{X}^{\underline{k},\underline{m}}_{<i},\boldsymbol{Y}^{\underline{k},\underline{m}}_{<i}\right)-H(\boldsymbol{X}^{\underline{k},\underline{m}}_i|\boldsymbol{X}^{\underline{k},\underline{m}}_{<i})=F_{k_i,m_i}(\mu_{\underline{k},\underline{m}}(i)),\]
where
\[\mu_{\underline{k},\underline{m}}(i):=\sum_{\boldsymbol{a}\in\prod_{\ell=1}^{i-1}\{0,1,\varepsilon_{\ell},\hdots,\varepsilon_{\ell}^{k_{\ell}-1},\zeta_{\ell},\hdots,\zeta_{\ell}^{m_{\ell}-1}\}}\mathbb{P}_{\boldsymbol{\cal{F}}_{\underline{k},\underline{m}}}\Big[\boldsymbol{X}^{\underline{k},\underline{m}}_{<i}=\boldsymbol{a}\Big]\delta_{k_i\mathbb{P}_{\boldsymbol{\cal{F}}_{\underline{k},\underline{m}}}\Big[\boldsymbol{X}^{\underline{k},\underline{m}}_i=0|\boldsymbol{X}^{\underline{k},\underline{m}}_{<i}=\boldsymbol{a}\Big]}.\]
\end{proposition}
First, note that for every $i$ and every $\boldsymbol{a}\in\prod_{\ell=1}^{i-1}\{0,1,\varepsilon_{\ell},\hdots,\varepsilon_{\ell}^{k_{\ell}-1},\zeta_{\ell},\hdots,\zeta_{\ell}^{m_{\ell}-1}\}$, we have

\begin{equation}\label{zero}\mathbb{P}_{\boldsymbol{\cal{F}}_{\underline{k},\underline{m}}}\Big[\boldsymbol{X}^{\underline{k},\underline{m}}_i=0|\boldsymbol{X}^{\underline{k},\underline{m}}_{<i}=\boldsymbol{a}\Big]=\mathbb{P}_{\boldsymbol{\cal{F}}_{\underline{k},\underline{m}}}\Big[\boldsymbol{X}^{\underline{k},\underline{m}}_i=\varepsilon_i|\boldsymbol{X}^{\underline{k},\underline{m}}_{<i}=\boldsymbol{a}\Big]=\hdots=\mathbb{P}_{\boldsymbol{\cal{F}}_{\underline{k},\underline{m}}}\Big[\boldsymbol{X}^{\underline{k},\underline{m}}_i=\varepsilon_i^{k_i-1}|\boldsymbol{X}^{\underline{k},\underline{m}}_{<i}=\boldsymbol{a}\Big]\end{equation}

and

\begin{equation}\label{one}\mathbb{P}_{\boldsymbol{\cal{F}}_{\underline{k},\underline{m}}}\Big[\boldsymbol{X}^{\underline{k},\underline{m}}_i=1|\boldsymbol{X}^{\underline{k},\underline{m}}_{<i}=\boldsymbol{a}\Big]=\mathbb{P}_{\boldsymbol{\cal{F}}_{\underline{k},\underline{m}}}\Big[\boldsymbol{X}^{\underline{k},\underline{m}}_i=\zeta_i|\boldsymbol{X}^{\underline{k},\underline{m}}_{<i}=\boldsymbol{a}\Big]=\hdots=\mathbb{P}_{\boldsymbol{\cal{F}}_{\underline{k},\underline{m}}}\Big[\boldsymbol{X}^{\underline{k},\underline{m}}_i=\zeta_i^{m_i-1}|\boldsymbol{X}^{\underline{k},\underline{m}}_{<i}=\boldsymbol{a}\Big].\end{equation}
From these, we have that the quantities in~\eqref{zero} are equal to

\[\frac{1-m_i\mathbb{P}_{\boldsymbol{\cal{F}}_{\underline{k},\underline{m}}}\Big[\boldsymbol{X}^{\underline{k},\underline{m}}_i=1|\boldsymbol{X}^{\underline{k},\underline{m}}_{<i}=\boldsymbol{a}\Big]}{k_i},\]

and the quantities in~\eqref{one} are equal to

\[\frac{1-k_i\mathbb{P}_{\boldsymbol{\cal{F}}_{\underline{k},\underline{m}}}\Big[\boldsymbol{X}^{\underline{k},\underline{m}}_i=0|\boldsymbol{X}^{\underline{k},\underline{m}}_{<i}=\boldsymbol{a}\Big]}{m_i}.\]
This simple observation will be important in our proof of the following two lemmas giving us Proposition~\ref{differenceprop}.
\begin{lemma}\label{HF}

\begin{equation}\label{HFformula}H(\boldsymbol{X}^{\underline{k},\underline{m}}_i|\boldsymbol{X}^{\underline{k},\underline{m}}_{<i})=\sum_{\boldsymbol{a}\in\prod_{\ell=1}^{i-1}\{0,1,\varepsilon_{\ell},\hdots,\varepsilon_{\ell}^{k_{\ell}-1},\zeta_{\ell},\hdots,\zeta_{\ell}^{m_{\ell}-1}\}}\mathbb{P}_{\boldsymbol{\cal{F}}_{\underline{k},\underline{m}}}\Big[\boldsymbol{X}^{\underline{k},\underline{m}}_{<i}=\boldsymbol{a}\Big]h_{k_i,m_i}\left(\mathbb{P}_{\boldsymbol{\cal{F}}_{\underline{k},\underline{m}}}\Big[\boldsymbol{X}^{\underline{k},\underline{m}}_i=0|\boldsymbol{X}^{\underline{k},\underline{m}}_{<i}=\boldsymbol{a}\Big]\right),\end{equation}

\end{lemma}
\begin{proof}By definition, we have the following chain of equalities.
\begin{small}
\begin{eqnarray*}
&&H(\boldsymbol{X}^{\underline{k},\underline{m}}_i|\boldsymbol{X}^{\underline{k},\underline{m}}_{<i})\\
&=&\sum_{\boldsymbol{a}\in\prod_{\ell=1}^{i-1}\{0,1,\varepsilon_{\ell},\hdots,\varepsilon_{\ell}^{k_{\ell}-1},\zeta_{\ell},\hdots,\zeta_{\ell}^{m_{\ell}-1}\}}\mathbb{P}_{\boldsymbol{\cal{F}}_{\underline{k},\underline{m}}}\Big[\boldsymbol{X}^{\underline{k},\underline{m}}_{<i}=\boldsymbol{a}\Big]H(\boldsymbol{X}^{\underline{k},\underline{m}}_i|\boldsymbol{X}^{\underline{k},\underline{m}}_{<i}=\boldsymbol{a})\\
&=&-\sum_{\boldsymbol{a}\in\prod_{\ell=1}^{i-1}\{0,1,\varepsilon_{\ell},\hdots,\varepsilon_{\ell}^{k_{\ell}-1},\zeta_{\ell},\hdots,\zeta_{\ell}^{m_{\ell}-1}\}}\mathbb{P}_{\boldsymbol{\cal{F}}_{\underline{k},\underline{m}}}\Big[\boldsymbol{X}^{\underline{k},\underline{m}}_{<i}=\boldsymbol{a}\Big]\Bigg(\sum_{j=1}^{k_i}\mathbb{P}_{\boldsymbol{\cal{F}}_{\underline{k},\underline{m}}}\Big[\boldsymbol{X}^{\underline{k},\underline{m}}_i=\varepsilon_i^j|\boldsymbol{X}^{\underline{k},\underline{m}}_{<i}=\boldsymbol{a}\Big]\log\left(\mathbb{P}_{\boldsymbol{\cal{F}}_{\underline{k},\underline{m}}}\Big[\boldsymbol{X}^{\underline{k},\underline{m}}_i=\varepsilon_i^j|\boldsymbol{X}^{\underline{k},\underline{m}}_{<i}=\boldsymbol{a}\Big]\right)\\
&&+\sum_{\ell=0}^{m_i-1}\mathbb{P}_{\boldsymbol{\cal{F}}_{\underline{k},\underline{m}}}\Big[\boldsymbol{X}^{\underline{k},\underline{m}}_i=\zeta_i^{\ell}|\boldsymbol{X}^{\underline{k},\underline{m}}_{<i}=\boldsymbol{a}\Big]\log\left(\mathbb{P}_{\boldsymbol{\cal{F}}_{\underline{k},\underline{m}}}\Big[\boldsymbol{X}^{\underline{k},\underline{m}}_i=\zeta_i^{\ell}|\boldsymbol{X}^{\underline{k},\underline{m}}_{<i}=\boldsymbol{a}\Big]\right)\Bigg)\\
&=&-\sum_{\boldsymbol{a}\in\prod_{\ell=1}^{i-1}\{0,1,\varepsilon_{\ell},\hdots,\varepsilon_{\ell}^{k_{\ell}-1},\zeta_{\ell},\hdots,\zeta_{\ell}^{m_{\ell}-1}\}}\mathbb{P}_{\boldsymbol{\cal{F}}_{\underline{k},\underline{m}}}\Big[\boldsymbol{X}^{\underline{k},\underline{m}}_{<i}=\boldsymbol{a}\Big]\Bigg(k_i\mathbb{P}_{\boldsymbol{\cal{F}}_k}\Big[\boldsymbol{X}^{\underline{k},\underline{m}}_i=0|\boldsymbol{X}^{\underline{k},\underline{m}}_{<i}=\boldsymbol{a}\Big]\log\mathbb{P}_{\boldsymbol{\cal{F}}_{\underline{k},\underline{m}}}\Big[\boldsymbol{X}^{\underline{k},\underline{m}}_i=0|\boldsymbol{X}^{\underline{k},\underline{m}}_{<i}=\boldsymbol{a}\Big]\\
&&+\left(1-k_i\mathbb{P}_{\boldsymbol{\cal{F}}_{\underline{k},\underline{m}}}\Big[\boldsymbol{X}^{\underline{k},\underline{m}}_i=0|\boldsymbol{X}_{<i}=\boldsymbol{a}\Big]\right)\log\left(\frac{1-k_i\mathbb{P}_{\boldsymbol{\cal{F}}_{\underline{k},\underline{m}}}\Big[\boldsymbol{X}^{\underline{k},\underline{m}}_i=0|\boldsymbol{X}_{<i}=\boldsymbol{a}\Big]}{m_i}\right)\Bigg)\\
&=&\sum_{\boldsymbol{a}\in\prod_{\ell=1}^{i-1}\{0,1,\varepsilon_{\ell},\hdots,\varepsilon_{\ell}^{k_{\ell}-1},\zeta_{\ell},\hdots,\zeta_{\ell}^{m_{\ell}-1}\}}\mathbb{P}_{\boldsymbol{\cal{F}}_{\underline{k},\underline{m}}}\Big[\boldsymbol{X}^{\underline{k},\underline{m}}_{<i}=\boldsymbol{a}\Big]h_{k_i,m_i}\left(\mathbb{P}_{\boldsymbol{\cal{F}}_{\underline{k},\underline{m}}}\Big[\boldsymbol{X}^{\underline{k},\underline{m}}_i=0|\boldsymbol{X}^{\underline{k},\underline{m}}_{<i}=\boldsymbol{a}\Big]\right),
\end{eqnarray*}
\end{small}
as required.
\end{proof}
We also have the following lemma.
\begin{lemma}\label{HFmin}For every $i$, we have

\begin{eqnarray*}
&&H\left(\boldsymbol{X}^{\underline{k},\underline{m}}_i\boldsymbol{Y}^{\underline{k},\underline{m}}_i|\boldsymbol{X}^{\underline{k},\underline{m}}_{<i},\boldsymbol{Y}^{\underline{k},\underline{m}}_{<i}\right)\\&=&\sum_{\boldsymbol{a},\boldsymbol{b}\in\prod_{\ell=1}^{i-1}\{0,1,\varepsilon_{\ell},\hdots,\varepsilon_{\ell}^{k_{\ell}-1},\zeta_{\ell},\hdots,\zeta_{\ell}^{m_{\ell}-1}\}}\mathbb{P}_{\boldsymbol{\cal{F}}_{\underline{k},\underline{m}}}\Big[\boldsymbol{X}^{\underline{k},\underline{m}}_{<i}=\boldsymbol{a},\boldsymbol{Y}^{\underline{k},\underline{m}}_{<i}=\boldsymbol{b}\Big]g_{k_i,m_i}\left(\mathbb{P}_{\boldsymbol{\cal{F}}_{\underline{k},\underline{m}}}\Big[\boldsymbol{X}^{\underline{k},\underline{m}}_i=0|\boldsymbol{X}^{\underline{k},\underline{m}}_{<i}=\boldsymbol{a}\Big],\mathbb{P}_{\boldsymbol{\cal{F}}_{\underline{k},\underline{m}}}\Big[\boldsymbol{Y}^{\underline{k},\underline{m}}_i=0|\boldsymbol{Y}^{\underline{k},\underline{m}}_{<i}=\boldsymbol{b}\Big]\right),
\end{eqnarray*}

\end{lemma}
\begin{proof}By the definition of conditional Shannon entropy, we have
\begin{small}
\begin{eqnarray*}
H\left(\boldsymbol{X}^{\underline{k},\underline{m}}_i\boldsymbol{Y}^{\underline{k},\underline{m}}_i|\boldsymbol{X}^{\underline{k},\underline{m}}_{<i},\boldsymbol{Y}^{\underline{k},\underline{m}}_{<i}\right)&=&\sum_{\boldsymbol{a},\boldsymbol{b}\in\prod_{\ell=1}^{i-1}\{0,1,\varepsilon_{\ell},\hdots,\varepsilon_{\ell}^{k_{\ell}-1},\zeta_{\ell},\hdots,\zeta_{\ell}^{m_{\ell}-1}\}}\mathbb{P}_{\boldsymbol{\cal{F}}_k}\Big[\boldsymbol{X}^{\underline{k},\underline{m}}_{<i}=\boldsymbol{a},\boldsymbol{Y}^{\underline{k},\underline{m}}_{<i}=\boldsymbol{b}\Big]H\left(\boldsymbol{X}^{\underline{k},\underline{m}}_i\boldsymbol{Y}^{\underline{k},\underline{m}}_i|\boldsymbol{X}^{\underline{k},\underline{m}}_{<i}=\boldsymbol{a},\boldsymbol{Y}^{\underline{k},\underline{m}}_{<i}=\boldsymbol{b}\right).
\end{eqnarray*}
\end{small}
Moreover, for every $\boldsymbol{a},\boldsymbol{b}\in\prod_{\ell=1}^{i-1}\{0,1,\varepsilon_{\ell},\hdots,\varepsilon_{\ell}^{k_{\ell}-1},\zeta_{\ell},\hdots,\zeta_{\ell}^{m_{\ell}-1}\}$ we have
\begin{small}
\begin{eqnarray*}
&&H\left(\boldsymbol{X}^{\underline{k},\underline{m}}_i\boldsymbol{Y}^{\underline{k},\underline{m}}_i|\boldsymbol{X}^{\underline{k},\underline{m}}_{<i}=\boldsymbol{a},\boldsymbol{Y}^{\underline{k},\underline{m}}_{<i}=\boldsymbol{b}\right)\\&=&-\sum_{j=1}^{k_i}\mathbb{P}_{\boldsymbol{\cal{F}}_{\underline{k},\underline{m}}}\Big[\boldsymbol{X}^{\underline{k},\underline{m}}_i\boldsymbol{Y}^{\underline{k},\underline{m}}_i=\varepsilon_i^j|\boldsymbol{X}^{\underline{k},\underline{m}}_{<i}=\boldsymbol{a},\boldsymbol{Y}^{\underline{k},\underline{m}}_{<i}=\boldsymbol{b}\Big]\log\left(\mathbb{P}_{\boldsymbol{\cal{F}}_{\underline{k},\underline{m}}}\Big[\boldsymbol{X}^{\underline{k},\underline{m}}_i\boldsymbol{Y}^{\underline{k},\underline{m}}_i=\varepsilon_i^j|\boldsymbol{X}^{\underline{k},\underline{m}}_{<i}=\boldsymbol{a},\boldsymbol{Y}^{\underline{k},\underline{m}}_{<i}=\boldsymbol{b}\Big]\right)\\
&&-\sum_{\ell=0}^{m_i-1}\mathbb{P}_{\boldsymbol{\cal{F}}_{\underline{k},\underline{m}}}\Big[\boldsymbol{X}^{\underline{k},\underline{m}}_i\boldsymbol{Y}^{\underline{k},\underline{m}}_i=\zeta_i^{\ell}|\boldsymbol{X}^{\underline{k},\underline{m}}_{<i}=\boldsymbol{a},\boldsymbol{Y}^{\underline{k},\underline{m}}_{<i}=\boldsymbol{b}\Big]\log\left(\mathbb{P}_{\boldsymbol{\cal{F}}_{\underline{k},\underline{m}}}\Big[\boldsymbol{X}^{\underline{k},\underline{m}}_i\boldsymbol{Y}^{\underline{k},\underline{m}}_i=\zeta_i^{\ell}|\boldsymbol{X}^{\underline{k},\underline{m}}_{<i}=\boldsymbol{a},\boldsymbol{Y}^{\underline{k},\underline{m}}_{<i}=\boldsymbol{b}\Big]\right)
\end{eqnarray*}
\end{small}
Calculating probabilities, we have for $0<j<k_i$
\begin{small}
\begin{eqnarray*}
&&\mathbb{P}_{\boldsymbol{\cal{F}}_{\underline{k},\underline{m}}}\Big[\boldsymbol{X}^{\underline{k},\underline{m}}_i\boldsymbol{Y}^{\underline{k},\underline{m}}_i=\varepsilon_i^j|\boldsymbol{X}^{\underline{k},\underline{m}}_{<i}=\boldsymbol{a},\boldsymbol{Y}^{\underline{k},\underline{m}}_{<i}=\boldsymbol{b}\Big]\\
&=&\sum_{\ell=0}^{m_i-1}\mathbb{P}_{\boldsymbol{\cal{F}}_{\underline{k},\underline{m}}}\Big[\boldsymbol{X}^{\underline{k},\underline{m}}_i=\zeta_i^{\ell}|\boldsymbol{X}^{\underline{k},\underline{m}}_{<i}=\boldsymbol{a}\Big]\mathbb{P}_{\boldsymbol{\cal{F}}_{\underline{k},\underline{m}}}\Big[\boldsymbol{Y}^{\underline{k},\underline{m}}_i=\varepsilon_i^j|\boldsymbol{Y}^{\underline{k},\underline{m}}_{<i}=\boldsymbol{b}\Big]+\sum_{\ell=0}^{m_i-1}\mathbb{P}_{\boldsymbol{\cal{F}}_{\underline{k},\underline{m}}}\Big[\boldsymbol{X}^{\underline{k},\underline{m}}_i=\varepsilon_i^j|\boldsymbol{X}^{\underline{k},\underline{m}}_{<i}=\boldsymbol{a}\Big]\mathbb{P}_{\boldsymbol{\cal{F}}_{\underline{k},\underline{m}}}\Big[\boldsymbol{Y}^{\underline{k},\underline{m}}_i=\zeta_i^{\ell}|\boldsymbol{Y}^{\underline{k},\underline{m}}_{<i}=\boldsymbol{b}\Big]\\
&&+\sum_{\substack{0< u,v<k_i\\u+v=j}}\mathbb{P}_{\boldsymbol{\cal{F}}_{\underline{k},\underline{m}}}\Big[\boldsymbol{X}^{\underline{k},\underline{m}}_i=\varepsilon_i^u|\boldsymbol{X}^{\underline{k},\underline{m}}_{<i}=\boldsymbol{a}\Big]\mathbb{P}_{\boldsymbol{\cal{F}}_{\underline{k},\underline{m}}}\Big[\boldsymbol{Y}^{\underline{k},\underline{m}}_i=\varepsilon_i^v|\boldsymbol{Y}^{\underline{k},\underline{m}}_{<i}=\boldsymbol{b}\Big]\\
&=&m_i\mathbb{P}_{\boldsymbol{\cal{F}}_{\underline{k},\underline{m}}}\Big[\boldsymbol{X}^{\underline{k},\underline{m}}_i=1|\boldsymbol{X}^{\underline{k},\underline{m}}_{<i}=\boldsymbol{a}\Big]\mathbb{P}_{\boldsymbol{\cal{F}}_{\underline{k},\underline{m}}}\Big[\boldsymbol{Y}^{\underline{k},\underline{m}}_i=\varepsilon_i^j|\boldsymbol{Y}^{\underline{k},\underline{m}}_{<i}=\boldsymbol{b}\Big]+m_i\mathbb{P}_{\boldsymbol{\cal{F}}_{\underline{k},\underline{m}}}\Big[\boldsymbol{X}^{\underline{k},\underline{m}}_i=\varepsilon_i^j|\boldsymbol{X}^{\underline{k},\underline{m}}_{<i}=\boldsymbol{a}\Big]\mathbb{P}_{\boldsymbol{\cal{F}}_{\underline{k},\underline{m}}}\Big[\boldsymbol{Y}^{\underline{k},\underline{m}}_i=1|\boldsymbol{Y}^{\underline{k},\underline{m}}_{<i}=\boldsymbol{b}\Big]\\
&&+(j-1)\mathbb{P}_{\boldsymbol{\cal{F}}_{\underline{k},\underline{m}}}\Big[\boldsymbol{X}^{\underline{k},\underline{m}}_i=0|\boldsymbol{X}^{\underline{k},\underline{m}}_{<i}=\boldsymbol{a}\Big]\mathbb{P}_{\boldsymbol{\cal{F}}_{\underline{k},\underline{m}}}\Big[\boldsymbol{Y}^{\underline{k},\underline{m}}_i=0|\boldsymbol{Y}^{\underline{k},\underline{m}}_{<i}=\boldsymbol{b}\Big]\\
&=&\left(1-k_i\mathbb{P}_{\boldsymbol{\cal{F}}_{\underline{k},\underline{m}}}\Big[\boldsymbol{X}^{\underline{k},\underline{m}}_i=0|\boldsymbol{X}^{\underline{k},\underline{m}}_{<i}=\boldsymbol{a}\Big]\right)\mathbb{P}_{\boldsymbol{\cal{F}}_{\underline{k},\underline{m}}}\Big[\boldsymbol{Y}^{\underline{k},\underline{m}}_i=0|\boldsymbol{Y}^{\underline{k},\underline{m}}_{<i}=\boldsymbol{b}\Big]+\mathbb{P}_{\boldsymbol{\cal{F}}_{\underline{k},\underline{m}}}\Big[\boldsymbol{X}^{\underline{k},\underline{m}}_i=0|\boldsymbol{X}^{\underline{k},\underline{m}}_{<i}=\boldsymbol{a}\Big]\left(1-k_i\mathbb{P}_{\boldsymbol{\cal{F}}_{\underline{k},\underline{m}}}\Big[\boldsymbol{Y}^{\underline{k},\underline{m}}_i=0|\boldsymbol{Y}^{\underline{k},\underline{m}}_{<i}=\boldsymbol{b}\Big]\right)\\
&&+(j-1)\mathbb{P}_{\boldsymbol{\cal{F}}_{\underline{k},\underline{m}}}\Big[\boldsymbol{X}^{\underline{k},\underline{m}}_i=0|\boldsymbol{X}^{\underline{k},\underline{m}}_{<i}=\boldsymbol{a}\Big]\mathbb{P}_{\boldsymbol{\cal{F}}_{\underline{k},\underline{m}}}\Big[\boldsymbol{Y}^{\underline{k},\underline{m}}_i=0|\boldsymbol{Y}^{\underline{k},\underline{m}}_{<i}=\boldsymbol{b}\Big]\\
&=&\mathbb{P}_{\boldsymbol{\cal{F}}_{\underline{k},\underline{m}}}\Big[\boldsymbol{X}^{\underline{k},\underline{m}}_i=0|\boldsymbol{X}^{\underline{k},\underline{m}}_{<i}=\boldsymbol{a}\Big]+\mathbb{P}_{\boldsymbol{\cal{F}}_{\underline{k},\underline{m}}}\Big[\boldsymbol{Y}^{\underline{k},\underline{m}}_i=0|\boldsymbol{Y}^{\underline{k},\underline{m}}_{<i}=\boldsymbol{b}\Big]-(2k_i-j+1)\mathbb{P}_{\boldsymbol{\cal{F}}_{\underline{k},\underline{m}}}\Big[\boldsymbol{X}^{\underline{k},\underline{m}}_i=0|\boldsymbol{X}^{\underline{k},\underline{m}}_{<i}=\boldsymbol{a}\Big]\mathbb{P}_{\boldsymbol{\cal{F}}_{\underline{k},\underline{m}}}\Big[\boldsymbol{Y}^{\underline{k},\underline{m}}_i=0|\boldsymbol{Y}^{\underline{k},\underline{m}}_{<i}=\boldsymbol{b}\Big].
\end{eqnarray*}
\end{small}
Furthermore,
\begin{small}
\begin{eqnarray*}
&&\mathbb{P}_{\boldsymbol{\cal{F}}_{\underline{k},\underline{m}}}\Big[\boldsymbol{X}^{\underline{k},\underline{m}}_i\boldsymbol{Y}^{\underline{k},\underline{m}}_i=0|\boldsymbol{X}^{\underline{k},\underline{m}}_{<i}=\boldsymbol{a},\boldsymbol{Y}^{\underline{k},\underline{m}}_{<i}=\boldsymbol{b}\Big]\\
&=& \mathbb{P}_{\boldsymbol{\cal{F}}_{\underline{k},\underline{m}}}\Big[\boldsymbol{X}^{\underline{k},\underline{m}}_i=0|\boldsymbol{X}^{\underline{k},\underline{m}}_{<i}=\boldsymbol{a}\Big]+\mathbb{P}_{\boldsymbol{\cal{F}}_{\underline{k},\underline{m}}}\Big[\boldsymbol{Y}^{\underline{k},\underline{m}}_i=0|\boldsymbol{Y}^{\underline{k},\underline{m}}_{<i}=\boldsymbol{b}\Big]-\mathbb{P}_{\boldsymbol{\cal{F}}_{\underline{k},\underline{m}}}\Big[\boldsymbol{X}^{\underline{k},\underline{m}}_i=0|\boldsymbol{X}^{\underline{k},\underline{m}}_{<i}=\boldsymbol{a}\Big]\mathbb{P}_{\boldsymbol{\cal{F}}_{\underline{k},\underline{m}}}\Big[\boldsymbol{Y}^{\underline{k},\underline{m}}_i=0|\boldsymbol{Y}^{\underline{k},\underline{m}}_{<i}=\boldsymbol{b}\Big]
\\&&+\sum_{\substack{0< u,v< k_i\\ u+v\geq k_i}}\mathbb{P}_{\boldsymbol{\cal{F}}_{\underline{k},\underline{m}}}\Big[\boldsymbol{X}^{\underline{k},\underline{m}}_i=\varepsilon_i^u|\boldsymbol{X}^{\underline{k},\underline{m}}_{<i}=\boldsymbol{a}\Big]\mathbb{P}_{\boldsymbol{\cal{F}}_{\underline{k},\underline{m}}}\Big[\boldsymbol{Y}^{\underline{k},\underline{m}}_i=\varepsilon_i^v|\boldsymbol{Y}^{\underline{k},\underline{m}}_{<i}=\boldsymbol{b}\Big]\\
&=&\mathbb{P}_{\boldsymbol{\cal{F}}_{\underline{k},\underline{m}}}\Big[\boldsymbol{X}^{\underline{k},\underline{m}}_i=0|\boldsymbol{X}^{\underline{k},\underline{m}}_{<i}=\boldsymbol{a}\Big]+\mathbb{P}_{\boldsymbol{\cal{F}}_{\underline{k},\underline{m}}}\Big[\boldsymbol{Y}^{\underline{k},\underline{m}}_i=0|\boldsymbol{Y}^{\underline{k},\underline{m}}_{<i}=\boldsymbol{b}\Big]-\mathbb{P}_{\boldsymbol{\cal{F}}_{\underline{k},\underline{m}}}\Big[\boldsymbol{X}^{\underline{k},\underline{m}}_i=0|\boldsymbol{X}^{\underline{k},\underline{m}}_{<i}=\boldsymbol{a}\Big]\mathbb{P}_{\boldsymbol{\cal{F}}_{\underline{k},\underline{m}}}\Big[\boldsymbol{Y}^{\underline{k},\underline{m}}_i=0|\boldsymbol{Y}^{\underline{k},\underline{m}}_{<i}=\boldsymbol{b}\Big]\\
&&+\sum_{j=1}^{k_i-1}\mathbb{P}_{\boldsymbol{\cal{F}}_{\underline{k},\underline{m}}}\Big[\boldsymbol{X}^{\underline{k},\underline{m}}_i=\varepsilon_i^j|\boldsymbol{X}^{\underline{k},\underline{m}}_{<i}=\boldsymbol{a}\Big]\sum_{\ell=k_i-j}^{k_i-1}\mathbb{P}_{\boldsymbol{\cal{F}}_{\underline{k},\underline{m}}}\Big[\boldsymbol{Y}^{\underline{k},\underline{m}}_i=\varepsilon_i^{\ell}|\boldsymbol{Y}^{\underline{k},\underline{m}}_{<i}=\boldsymbol{b}\Big]\\
&=&\mathbb{P}_{\boldsymbol{\cal{F}}_{\underline{k},\underline{m}}}\Big[\boldsymbol{X}^{\underline{k},\underline{m}}_i=0|\boldsymbol{X}^{\underline{k},\underline{m}}_{<i}=\boldsymbol{a}\Big]+\mathbb{P}_{\boldsymbol{\cal{F}}_{\underline{k},\underline{m}}}\Big[\boldsymbol{Y}^{\underline{k},\underline{m}}_i=0|\boldsymbol{Y}^{\underline{k},\underline{m}}_{<i}=\boldsymbol{b}\Big]-\mathbb{P}_{\boldsymbol{\cal{F}}_{\underline{k},\underline{m}}}\Big[\boldsymbol{X}^{\underline{k},\underline{m}}_i=0|\boldsymbol{X}^{\underline{k},\underline{m}}_{<i}=\boldsymbol{a}\Big]\mathbb{P}_{\boldsymbol{\cal{F}}_{\underline{k},\underline{m}}}\Big[\boldsymbol{Y}^{\underline{k},\underline{m}}_i=0|\boldsymbol{Y}^{\underline{k},\underline{m}}_{<i}=\boldsymbol{b}\Big]\\
&&+\frac{k_i(k_i-1)}{2}\mathbb{P}_{\boldsymbol{\cal{F}}_{\underline{k},\underline{m}}}\Big[\boldsymbol{X}^{\underline{k},\underline{m}}_i=0|\boldsymbol{X}^{\underline{k},\underline{m}}_{<i}=\boldsymbol{a}\Big]\mathbb{P}_{\boldsymbol{\cal{F}}_{\underline{k},\underline{m}}}\Big[\boldsymbol{Y}^{\underline{k},\underline{m}}_i=0|\boldsymbol{Y}^{\underline{k},\underline{m}}_{<i}=\boldsymbol{b}\Big].
\end{eqnarray*}
\end{small}
We also have for every $0\leq \ell\leq m_i-1$,
\begin{small}
\begin{eqnarray*}
&&\mathbb{P}_{\boldsymbol{\cal{F}}_{\underline{k},\underline{m}}}\Big[\boldsymbol{X}^{\underline{k},\underline{m}}_i\boldsymbol{Y}^{\underline{k},\underline{m}}_i=\zeta_i^{\ell}|\boldsymbol{X}^{\underline{k},\underline{m}}_{<i}=\boldsymbol{a},\boldsymbol{Y}^{\underline{k},\underline{m}}_{<i}=\boldsymbol{b}\Big]\\
&=&\sum_{\substack{0\leq u,v\leq m_i-1\\ u+v\equiv\ell\bmod m_i}}\mathbb{P}_{\boldsymbol{\cal{F}}_{\underline{k},\underline{m}}}\Big[\boldsymbol{X}^{\underline{k},\underline{m}}_i=\zeta_i^u|\boldsymbol{X}^{\underline{k},\underline{m}}_{<i}=\boldsymbol{a}\Big]\mathbb{P}_{\boldsymbol{\cal{F}}_{\underline{k},\underline{m}}}\Big[\boldsymbol{Y}^{\underline{k},\underline{m}}_i=\zeta_i^v|\boldsymbol{Y}^{\underline{k},\underline{m}}_{<i}=\boldsymbol{b}\Big]\\
&=&\sum_{\substack{0\leq u,v\leq m_i-1\\ u+v=\ell\textup{ or }u+v=m_i+\ell}}\mathbb{P}_{\boldsymbol{\cal{F}}_{\underline{k},\underline{m}}}\Big[\boldsymbol{X}^{\underline{k},\underline{m}}_i=\zeta_i^u|\boldsymbol{X}^{\underline{k},\underline{m}}_{<i}=\boldsymbol{a}\Big]\mathbb{P}_{\boldsymbol{\cal{F}}_{\underline{k},\underline{m}}}\Big[\boldsymbol{Y}^{\underline{k},\underline{m}}_i=\zeta_i^v|\boldsymbol{Y}^{\underline{k},\underline{m}}_{<i}=\boldsymbol{b}\Big]\\
&=&\sum_{u=0}^{\ell}\mathbb{P}_{\boldsymbol{\cal{F}}_{\underline{k},\underline{m}}}\Big[\boldsymbol{X}^{\underline{k},\underline{m}}_i=\zeta_i^u|\boldsymbol{X}^{\underline{k},\underline{m}}_{<i}=\boldsymbol{a}\Big]\mathbb{P}_{\boldsymbol{\cal{F}}_{\underline{k},\underline{m}}}\Big[\boldsymbol{Y}^{\underline{k},\underline{m}}_i=\zeta_i^{\ell-u}|\boldsymbol{Y}^{\underline{k},\underline{m}}_{<i}=\boldsymbol{b}\Big]+\sum_{u=\ell+1}^{m_i-1}\mathbb{P}_{\boldsymbol{\cal{F}}_{\underline{k},\underline{m}}}\Big[\boldsymbol{X}^{\underline{k},\underline{m}}_i=\zeta_i^u|\boldsymbol{X}^{\underline{k},\underline{m}}_{<i}=\boldsymbol{a}\Big]\mathbb{P}_{\boldsymbol{\cal{F}}_{\underline{k},\underline{m}}}\Big[\boldsymbol{Y}^{\underline{k},\underline{m}}_i=\zeta_i^{m_i+\ell-u}|\boldsymbol{Y}^{\underline{k},\underline{m}}_{<i}=\boldsymbol{b}\Big]\\
&=&\sum_{u=0}^{\ell}\mathbb{P}_{\boldsymbol{\cal{F}}_{\underline{k},\underline{m}}}\Big[\boldsymbol{X}^{\underline{k},\underline{m}}_i=1|\boldsymbol{X}^{\underline{k},\underline{m}}_{<i}=\boldsymbol{a}\Big]\mathbb{P}_{\boldsymbol{\cal{F}}_{\underline{k},\underline{m}}}\Big[\boldsymbol{Y}^{\underline{k},\underline{m}}_i=1|\boldsymbol{Y}^{\underline{k},\underline{m}}_{<i}=\boldsymbol{b}\Big]+\sum_{u=\ell+1}^{m_i-1}\mathbb{P}_{\boldsymbol{\cal{F}}_{\underline{k},\underline{m}}}\Big[\boldsymbol{X}^{\underline{k},\underline{m}}_i=1|\boldsymbol{X}^{\underline{k},\underline{m}}_{<i}=\boldsymbol{a}\Big]\mathbb{P}_{\boldsymbol{\cal{F}}_{\underline{k},\underline{m}}}\Big[\boldsymbol{Y}^{\underline{k},\underline{m}}_i=1|\boldsymbol{Y}^{\underline{k},\underline{m}}_{<i}=\boldsymbol{b}\Big]\\
&=&\frac{1}{m_i}\left(m_i\mathbb{P}_{\boldsymbol{\cal{F}}_{\underline{k},\underline{m}}}\Big[\boldsymbol{X}^{\underline{k},\underline{m}}_i=1|\boldsymbol{X}^{\underline{k},\underline{m}}_{<i}=\boldsymbol{a}\Big]\right)\left(m_i\mathbb{P}_{\boldsymbol{\cal{F}}_{\underline{k},\underline{m}}}\Big[\boldsymbol{Y}^{\underline{k},\underline{m}}_i=1|\boldsymbol{Y}^{\underline{k},\underline{m}}_{<i}=\boldsymbol{b}\Big]\right)\\
&=&\frac{1}{m_i}\left(1-k_i\mathbb{P}_{\boldsymbol{\cal{F}}_{\underline{k},\underline{m}}}\Big[\boldsymbol{X}^{\underline{k},\underline{m}}_i=0|\boldsymbol{X}^{\underline{k},\underline{m}}_{<i}=\boldsymbol{a}\Big]\right)\left(1-k_i\mathbb{P}_{\boldsymbol{\cal{F}}_{\underline{k},\underline{m}}}\Big[\boldsymbol{Y}^{\underline{k},\underline{m}}_i=0|\boldsymbol{Y}^{\underline{k},\underline{m}}_{<i}=\boldsymbol{b}\Big]\right).
\end{eqnarray*}
\end{small}
We therefore have
\begin{small}
\begin{eqnarray*}
&&H\left(\boldsymbol{X}^{\underline{k},\underline{m}}_i\boldsymbol{Y}^{\underline{k},\underline{m}}_i|\boldsymbol{X}^{\underline{k},\underline{m}}_{<i}=\boldsymbol{a},\boldsymbol{Y}^{\underline{k},\underline{m}}_{<i}=\boldsymbol{b}\right)\\
&=&-\left(1-k_i\mathbb{P}_{\boldsymbol{\cal{F}}_{\underline{k},\underline{m}}}\Big[\boldsymbol{X}^{\underline{k},\underline{m}}_i=0|\boldsymbol{X}^{\underline{k},\underline{m}}_{<i}=\boldsymbol{a}\Big]\right)\left(1-k_i\mathbb{P}_{\boldsymbol{\cal{F}}_{\underline{k},\underline{m}}}\Big[\boldsymbol{Y}^{\underline{k},\underline{m}}_i=1|\boldsymbol{Y}^{\underline{k},\underline{m}}_{<i}=\boldsymbol{b}\Big]\right)\times\\
&\times&\log\left(\frac{1}{m_i}\left(1-k_i\mathbb{P}_{\boldsymbol{\cal{F}}_{\underline{k},\underline{m}}}\Big[\boldsymbol{X}^{\underline{k},\underline{m}}_i=0|\boldsymbol{X}^{\underline{k},\underline{m}}_{<i}=\boldsymbol{a}\Big]\right)\left(1-k_i\mathbb{P}_{\boldsymbol{\cal{F}}_{\underline{k},\underline{m}}}\Big[\boldsymbol{Y}^{\underline{k},\underline{m}}_i=0|\boldsymbol{Y}^{\underline{k},\underline{m}}_{<i}=\boldsymbol{b}\Big]\right)\right)\\
&&-\Bigg(\mathbb{P}_{\boldsymbol{\cal{F}}_{\underline{k},\underline{m}}}\Big[\boldsymbol{X}^{\underline{k},\underline{m}}_i=0|\boldsymbol{X}^{\underline{k},\underline{m}}_{<i}=\boldsymbol{a}\Big]+\mathbb{P}_{\boldsymbol{\cal{F}}_{\underline{k},\underline{m}}}\Big[\boldsymbol{Y}^{\underline{k},\underline{m}}_i=0|\boldsymbol{Y}^{\underline{k},\underline{m}}_{<i}=\boldsymbol{b}\Big]-\mathbb{P}_{\boldsymbol{\cal{F}}_{\underline{k},\underline{m}}}\Big[\boldsymbol{X}^{\underline{k},\underline{m}}_i=0|\boldsymbol{X}^{\underline{k},\underline{m}}_{<i}=\boldsymbol{a}\Big]\mathbb{P}_{\boldsymbol{\cal{F}}_{\underline{k},\underline{m}}}\Big[\boldsymbol{Y}^{\underline{k},\underline{m}}_i=0|\boldsymbol{Y}^{\underline{k},\underline{m}}_{<i}=\boldsymbol{b}\Big]\\&&+\frac{k_i(k_i-1)}{2}\mathbb{P}_{\boldsymbol{\cal{F}}_{\underline{k},\underline{m}}}\Big[\boldsymbol{X}^{\underline{k},\underline{m}}_i=0|\boldsymbol{X}^{\underline{k},\underline{m}}_{<i}=\boldsymbol{a}\Big]\mathbb{P}_{\boldsymbol{\cal{F}}_{\underline{k},\underline{m}}}\Big[\boldsymbol{Y}^{\underline{k},\underline{m}}_i=0|\boldsymbol{Y}^{\underline{k},\underline{m}}_{<i}=\boldsymbol{b}\Big]\Bigg)\times\\
&&\times\log\Bigg(\mathbb{P}_{\boldsymbol{\cal{F}}_{\underline{k},\underline{m}}}\Big[\boldsymbol{X}^{\underline{k},\underline{m}}_i=0|\boldsymbol{X}^{\underline{k},\underline{m}}_{<i}=\boldsymbol{a}\Big]+\mathbb{P}_{\boldsymbol{\cal{F}}_{\underline{k},\underline{m}}}\Big[\boldsymbol{Y}^{\underline{k},\underline{m}}_i=0|\boldsymbol{Y}^{\underline{k},\underline{m}}_{<i}=\boldsymbol{b}\Big]-\mathbb{P}_{\boldsymbol{\cal{F}}_{\underline{k},\underline{m}}}\Big[\boldsymbol{X}^{\underline{k},\underline{m}}_i=0|\boldsymbol{X}^{\underline{k},\underline{m}}_{<i}=\boldsymbol{a}\Big]\mathbb{P}_{\boldsymbol{\cal{F}}_{\underline{k},\underline{m}}}\Big[\boldsymbol{Y}^{\underline{k},\underline{m}}_i=0|\boldsymbol{Y}^{\underline{k},\underline{m}}_{<i}=\boldsymbol{b}\Big]\\&&+\frac{k_i(k_i-1)}{2}\mathbb{P}_{\boldsymbol{\cal{F}}_{\underline{k},\underline{m}}}\Big[\boldsymbol{X}^{\underline{k},\underline{m}}_i=0|\boldsymbol{X}^{\underline{k},\underline{m}}_{<i}=\boldsymbol{a}\Big]\mathbb{P}_{\boldsymbol{\cal{F}}_{\underline{k},\underline{m}}}\Big[\boldsymbol{Y}^{\underline{k},\underline{m}}_i=0|\boldsymbol{Y}^{\underline{k},\underline{m}}_{<i}=\boldsymbol{b}\Big]\Bigg)\\
&&-\sum_{j=1}^{k_i-1}\Bigg(\mathbb{P}_{\boldsymbol{\cal{F}}_{\underline{k},\underline{m}}}\Big[\boldsymbol{X}^{\underline{k},\underline{m}}_i=0|\boldsymbol{X}^{\underline{k},\underline{m}}_{<i}=\boldsymbol{a}\Big]+\mathbb{P}_{\boldsymbol{\cal{F}}_{\underline{k},\underline{m}}}\Big[\boldsymbol{Y}^{\underline{k},\underline{m}}_i=0|\boldsymbol{Y}^{\underline{k},\underline{m}}_{<i}=\boldsymbol{b}\Big]-(2k_i-j+1)\mathbb{P}_{\boldsymbol{\cal{F}}_{\underline{k},\underline{m}}}\Big[\boldsymbol{X}^{\underline{k},\underline{m}}_i=0|\boldsymbol{X}^{\underline{k},\underline{m}}_{<i}=\boldsymbol{a}\Big]\mathbb{P}_{\boldsymbol{\cal{F}}_{\underline{k},\underline{m}}}\Big[\boldsymbol{Y}^{\underline{k},\underline{m}}_i=0|\boldsymbol{Y}^{\underline{k},\underline{m}}_{<i}=\boldsymbol{b}\Big]\Bigg)\times\\&&
\times\log\Bigg(\mathbb{P}_{\boldsymbol{\cal{F}}_{\underline{k},\underline{m}}}\Big[\boldsymbol{X}^{\underline{k},\underline{m}}_i=0|\boldsymbol{X}^{\underline{k},\underline{m}}_{<i}=\boldsymbol{a}\Big]+\mathbb{P}_{\boldsymbol{\cal{F}}_{\underline{k},\underline{m}}}\Big[\boldsymbol{Y}^{\underline{k},\underline{m}}_i=0|\boldsymbol{Y}^{\underline{k},\underline{m}}_{<i}=\boldsymbol{b}\Big]-(2k_i-j+1)\mathbb{P}_{\boldsymbol{\cal{F}}_{\underline{k},\underline{m}}}\Big[\boldsymbol{X}^{\underline{k},\underline{m}}_i=0|\boldsymbol{X}^{\underline{k},\underline{m}}_{<i}=\boldsymbol{a}\Big]\mathbb{P}_{\boldsymbol{\cal{F}}_{\underline{k},\underline{m}}}\Big[\boldsymbol{Y}^{\underline{k},\underline{m}}_i=0|\boldsymbol{Y}^{\underline{k},\underline{m}}_{<i}=\boldsymbol{b}\Big]\Bigg)
\end{eqnarray*}
\end{small}
We obtain from the above computations that
\begin{small}
\begin{eqnarray*}
&&H\left(\boldsymbol{X}^{\underline{k},\underline{m}}_i\boldsymbol{Y}^{\underline{k},\underline{m}}_i|\boldsymbol{X}^{\underline{k},\underline{m}}_{<i},\boldsymbol{Y}^{\underline{k},\underline{m}}_{<i}\right)\\&=&\sum_{\boldsymbol{a},\boldsymbol{b}\in\prod_{\ell=1}^{i-1}\{0,1,\varepsilon_{\ell},\hdots,\varepsilon_{\ell}^{k_{\ell}-1},\zeta_{\ell},\hdots,\zeta_{\ell}^{m_{\ell}-1}\}}\mathbb{P}_{\boldsymbol{\cal{F}}_{\underline{k},\underline{m}}}\Big[\boldsymbol{X}^{\underline{k},\underline{m}}_{<i}=\boldsymbol{a},\boldsymbol{Y}^{\underline{k},\underline{m}}_{<i}=\boldsymbol{b}\Big]g_{k_i,m_i}\left(\mathbb{P}_{\boldsymbol{\cal{F}}_{\underline{k},\underline{m}}}\Big[\boldsymbol{X}^{\underline{k},\underline{m}}_i=0|\boldsymbol{X}^{\underline{k},\underline{m}}_{<i}=\boldsymbol{a}\Big],\mathbb{P}_{\boldsymbol{\cal{F}}_{\underline{k},\underline{m}}}\Big[\boldsymbol{Y}^{\underline{k},\underline{m}}_i=0|\boldsymbol{Y}^{\underline{k},\underline{m}}_{<i}=\boldsymbol{b}\Big]\right).
\end{eqnarray*}
\end{small}
Combining the two above lemmas, it is clear the we obtain Proposition~\ref{differenceprop}.
\end{proof}
\section{Minimizers of $F_{k,m}$}\label{optimization}
By Proposition~\ref{differenceprop} in the previous section,
\[H(\boldsymbol{X}^{\underline{k},\underline{m}}_i\boldsymbol{Y}^{\underline{k},\underline{m}}_i|\boldsymbol{X}^{\underline{k},\underline{m}}_{<i},\boldsymbol{Y}^{\underline{k},\underline{m}}_{<i})-H(\boldsymbol{X}^{\underline{k},\underline{m}}_i|\boldsymbol{X}^{\underline{k},\underline{m}}_{<i})=F_{k_i,m_i}(\mu_{\underline{k},\underline{m}}(i)),\]
where for integers $k,m\geq 1$ and probability distributions $\mu$ on $[0,1]$,
\[F_{k,m}(\mu):=\mathop{\mathbb{E}}_{(x,y)\sim\mu\times\mu}\Big[g_{k,m}\left(\frac{x}{k},\frac{y}{k}\right)\Big]-\mathop{\mathbb{E}}_{x\sim\mu}\Big[h_{k,m}\left(\frac{x}{k}\right)\Big],\]
and
\[\mu_{\underline{k},\underline{m}}(i):=\sum_{\boldsymbol{a}\in\prod_{\ell=1}^{i-1}\{0,1,\varepsilon_{\ell},\hdots,\varepsilon_{\ell}^{k_{\ell}-1},\zeta_{\ell},\hdots,\zeta_{\ell}^{m_{\ell}-1}\}}\mathbb{P}_{\boldsymbol{\cal{F}}_{\underline{k},\underline{m}}}\Big[\boldsymbol{X}^{\underline{k},\underline{m}}_{<i}=\boldsymbol{a}\Big]\delta_{k_i\mathbb{P}_{\boldsymbol{\cal{F}}_{\underline{k},\underline{m}}}\Big[\boldsymbol{X}^{\underline{k},\underline{m}}_i=0|\boldsymbol{X}^{\underline{k},\underline{m}}_{<i}=\boldsymbol{a}\Big]}.\]

Note that $\mathop{\mathbb{E}}_{x\sim\mu_{\underline{k},\underline{m}}(i)}=k_i\mathbb{P}_{\boldsymbol{\cal{F}}_{\underline{k},\underline{m}}}\Big[\boldsymbol{X}^{\underline{k},\underline{m}}_i=0\Big]$. Though we want to bound $F_{k_i,m_i}(\mu_{\underline{k},\underline{m}}(i))$ from below in Corollary~\ref{mainprop}, we do so by bounding from below $F_{k_i,m_i}(\mu)$ for every probability measure $\mu$ on $[0,1]$ satisfying $\mathop{\mathbb{E}}_{x\sim\mu}[x]=\phi$, where $0<\phi<\frac{1}{2}$. As explained at the end of the introduction, Corollary~\ref{mainprop} follows from Theorem~\ref{functionalthm} that we prove in this section.
\[F_{k,m}:\cal{M}_{\phi}\rightarrow\mathbb{R}\]
is a continuous functional on the space $\cal{M}_{\phi}$ of probability measures on $\left[0,1\right]$ with fixed $\mathop{\mathbb{E}}_{x\sim\mu}[x]=\phi$. 

$\cal{M}_{\phi}$ is a compact convex space. Since $F_{k,m}$ is, furthermore, continuous, $F_{k,m}$ has minimizers in $\cal{M}_{\phi}$. We show that $F_{k,m}$ is a concave functional on $\cal{M}_{\phi}$, and classify the form of its minimizers, using which we complete the proof of Theorem~\ref{functionalthm} in the rest of this paper.
\begin{proposition}\label{concavity}For every $k\geq 3$, $F_{k,m}:\cal{M}_{\phi}\rightarrow\mathbb{R}$ is concave.
\end{proposition}
\begin{proof}By definition,
\begin{small}
\begin{equation}\label{intF}
F_{k,m}(\mu)=\iint_{[0,1]^2}g_{k,m}\left(\frac{x}{k},\frac{y}{k}\right)\mu(dx)\mu(dy)-\int_0^1h_{k,m}\left(\frac{x}{k}\right)\mu(dx).
\end{equation}
\end{small}
Since
\begin{small}
\[-\int_0^1h_{k,m}\left(\frac{x}{k}\right)\mu(dx)\]
\end{small}
is linear in $\mu$, and so trivially concave, it suffices to show that
\begin{small}
\[\iint_{[0,1]^2}g_{k,m}\left(\frac{x}{k},\frac{y}{k}\right)\mu(dx)\mu(dy)\]
\end{small}
is concave in $\mu$. In order to show this, we use $\gamma(x):=\mu([0,x])$, the cumulative density function of $\mu$. Note that by integration by parts,
\begin{small}
\begin{equation}\label{gammaint}
\phi=\int_0^1x\mu(dx)=x\gamma(x)\Big|_0^1-\int_0^1\gamma(x)dx\implies \int_0^1\gamma(x)dx=1-\phi.
\end{equation}
\end{small}
Note that
\begin{small}
\[g_{k,m}\left(\frac{x}{k},\frac{y}{k}\right)=g_{k,1}\left(\frac{x}{k},\frac{y}{k}\right)+(1-x)(1-y)\log m.\]
\end{small}
Let $g_k:=g_{k,1}$. Since $\int_0^1x\mu(dx)=\phi$,
\begin{small}
\[\iint_{[0,1]^2}(1-x)(1-y)(\log m)\mu(dx)\mu(dy)=(1-\phi)^2\log m\]
\end{small}
is a constant, and so it suffices to show that
\begin{small}
\[\iint_{[0,1]^2}g_k\left(\frac{x}{k},\frac{y}{k}\right)\mu(dx)\mu(dy)\]
\end{small}
is concave in $\mu$. By integration by parts, we have
\begin{small}
\begin{eqnarray*}
&&\int_0^1g_{k}\left(\frac{x}{k},\frac{y}{k}\right)\mu(dx)\\
&=&g_{k}\left(\frac{x}{k},\frac{y}{k}\right)\gamma(x)\Big|_{x=0}^{x=1}-\int_0^1\frac{\partial}{\partial x}\left(g_{k}\left(\frac{x}{k},\frac{y}{k}\right)\right)\gamma(x)dx\\
&=&g_{k}\left(\frac{1}{k},\frac{y}{k}\right)-\int_0^1\gamma(x)(1-y)\left(1+\log((1-x)(1-y))\right)dx\\
&&+\frac{1}{k}\int_0^1\gamma(x)\left(1+\left(\frac{k(k-1)}{2}-1\right)\frac{y}{k}\right)\left(1+\log\left(\frac{x}{k}+\frac{y}{k}+\left(\frac{k(k-1)}{2}-1\right)\frac{xy}{k^2}\right)\right)dx\\
&&+\frac{1}{k}\sum_{j=1}^{k-1}\int_0^1\gamma(x)(1-(2k-j+1)\frac{y}{k})\left(1+\log(\frac{x}{k}+\frac{y}{k}-(2k-j+1)\frac{xy}{k^2})\right)dx.
\end{eqnarray*}
\end{small}
Applying integration by parts again,
\begin{small}
\begin{eqnarray*}
&&\iint_{[0,1]^2}g_k\left(\frac{x}{k},\frac{y}{k}\right)\mu(dx)\mu(dy)\\
&=&\int_0^1g_k\left(\frac{1}{k},\frac{y}{k}\right)\mu(dy)\\
&&-\int_0^1\gamma(x)\int_0^1(1-y)\left(1+\log((1-x)(1-y))\right)\mu(dy)dx\\
&&+\frac{1}{k}\int_0^1\gamma(x)\int_0^1\left(1+\left(\frac{k(k-1)}{2}-1\right)\frac{y}{k}\right)\left(1+\log\left(\frac{x}{k}+\frac{y}{k}+\left(\frac{k(k-1)}{2}-1\right)\frac{xy}{k^2}\right)\right)\mu(dy)dx\\
&&+\frac{1}{k}\sum_{j=1}^{k-1}\int_0^1\gamma(x)\int_0^1(1-(2k-j+1)\frac{y}{k})\left(1+\log(\frac{x}{k}+\frac{y}{k}-(2k-j+1)\frac{xy}{k^2})\right)\mu(dy)dx.
\end{eqnarray*}
\end{small}
The term
\begin{small}
\[\int_0^1g_k\left(\frac{1}{k},\frac{y}{k}\right)\mu(dy)\]
\end{small}
is linear in $\mu$. Note that
\begin{small}
\begin{eqnarray*}
&&\int_0^1(1-y)\left(1+\log((1-x)(1-y))\right)\mu(dy)\\
&=&(1-y)\left(1+\log((1-x)(1-y))\right)\gamma(y)\Bigg|_{y=0}^{y=1}-\int_0^1\gamma(y)\frac{\partial}{\partial y}\left((1-y)\left(1+\log((1-x)(1-y))\right)\right)dy\\
&=&\int_0^1\gamma(y)\left(2+\log((1-x)(1-y))\right)dy\\
&\implies& -\int_0^1\gamma(x)\int_0^1(1-y)\left(1+\log((1-x)(1-y))\right)\mu(dy)dx=-\int_0^1\int_0^1\gamma(x)\gamma(y)\left(2+\log((1-x)(1-y))\right)dxdy\\
&=&-2(1-\phi)^2-2(1-\phi)\int_0^1\gamma(x)\log(1-x)dx.
\end{eqnarray*}
\end{small}
This is also linear in $\mu$.
\begin{small}
\begin{eqnarray*}
&&\frac{1}{k}\int_0^1\left(1+\left(\frac{k(k-1)}{2}-1\right)\frac{y}{k}\right)\left(1+\log\left(\frac{x}{k}+\frac{y}{k}+\left(\frac{k(k-1)}{2}-1\right)\frac{xy}{k^2}\right)\right)\mu(dy)\\
&=&\frac{1}{k}\left(1+\left(\frac{k(k-1)}{2}-1\right)\frac{y}{k}\right)\left(1+\log\left(\frac{x}{k}+\frac{y}{k}+\left(\frac{k(k-1)}{2}-1\right)\frac{xy}{k^2}\right)\right)\gamma(y)\Bigg|_{y=0}^{y=1}\\&&-\frac{1}{k}\int_0^1\gamma(y)\frac{\partial}{\partial y}\left[\left(1+\left(\frac{k(k-1)}{2}-1\right)\frac{y}{k}\right)\left(1+\log\left(\frac{x}{k}+\frac{y}{k}+\left(\frac{k(k-1)}{2}-1\right)\frac{xy}{k^2}\right)\right)\right]dy\\
&=&\frac{1}{k}\left(1+\left(\frac{k(k-1)}{2}-1\right)\frac{1}{k}\right)\left(1+\log\left(\frac{x}{k}+\frac{1}{k}+\left(\frac{k(k-1)}{2}-1\right)\frac{x}{k^2}\right)\right)\\
&&-\frac{1}{k^2}\int_0^1\gamma(y)\left[\left(\frac{k(k-1)}{2}-1\right)\left(1+\log\left(\frac{x}{k}+\frac{y}{k}+\left(\frac{k(k-1)}{2}-1\right)\frac{xy}{k^2}\right)\right)+\frac{\left(1+\left(\frac{k(k-1)}{2}-1\right)\frac{x}{k}\right)\left(1+\left(\frac{k(k-1)}{2}-1\right)\frac{y}{k}\right)}{\frac{x}{k}+\frac{y}{k}+\left(\frac{k(k-1)}{2}-1\right)\frac{xy}{k^2}}\right]dy\\
&=&\frac{1}{k}\left(1+\left(\frac{k(k-1)}{2}-1\right)\frac{1}{k}\right)\left(1+\log\left(\frac{x}{k}+\frac{1}{k}+\left(\frac{k(k-1)}{2}-1\right)\frac{x}{k^2}\right)\right)\\
&&-\frac{1}{k^2}\int_0^1\gamma(y)\left(\frac{k(k-1)}{2}-1\right)\left(1+\log\left(\frac{\left(1+\left(\frac{k(k-1)}{2}-1\right)\frac{x}{k}\right)\left(1+\left(\frac{k(k-1)}{2}-1\right)\frac{y}{k}\right)-1}{\left(\frac{k(k-1)}{2}-1\right)}\right)\right)dy\\
&&-\frac{1}{k^2}\int_0^1\gamma(y)\left(\frac{k(k-1)}{2}-1\right)\frac{\left(1+\left(\frac{k(k-1)}{2}-1\right)\frac{x}{k}\right)\left(1+\left(\frac{k(k-1)}{2}-1\right)\frac{y}{k}\right)}{\left(1+\left(\frac{k(k-1)}{2}-1\right)\frac{x}{k}\right)\left(1+\left(\frac{k(k-1)}{2}-1\right)\frac{y}{k}\right)-1}dy\\
&\implies&\frac{1}{k}\int_0^1\gamma(x)\int_0^1\left(1+\left(\frac{k(k-1)}{2}-1\right)\frac{y}{k}\right)\left(1+\log\left(\frac{x}{k}+\frac{y}{k}+\left(\frac{k(k-1)}{2}-1\right)\frac{xy}{k^2}\right)\right)\mu(dy)dx\\
&=&\frac{1}{k}\int_0^1\gamma(x)\left(1+\left(\frac{k(k-1)}{2}-1\right)\frac{1}{k}\right)\left(1+\log\left(\frac{x}{k}+\frac{1}{k}+\left(\frac{k(k-1)}{2}-1\right)\frac{x}{k^2}\right)\right)dx\\
&&-\frac{1}{k^2}\int_0^1\int_0^1\gamma(x)\gamma(y)\left(\frac{k(k-1)}{2}-1\right)\left(1+\log\left(\frac{\left(1+\left(\frac{k(k-1)}{2}-1\right)\frac{x}{k}\right)\left(1+\left(\frac{k(k-1)}{2}-1\right)\frac{y}{k}\right)-1}{\left(\frac{k(k-1)}{2}-1\right)}\right)\right)dxdy\\
&&-\frac{1}{k^2}\int_0^1\int_0^1\gamma(x)\gamma(y)\left(\frac{k(k-1)}{2}-1\right)\frac{\left(1+\left(\frac{k(k-1)}{2}-1\right)\frac{x}{k}\right)\left(1+\left(\frac{k(k-1)}{2}-1\right)\frac{y}{k}\right)}{\left(1+\left(\frac{k(k-1)}{2}-1\right)\frac{x}{k}\right)\left(1+\left(\frac{k(k-1)}{2}-1\right)\frac{y}{k}\right)-1}dxdy.
\end{eqnarray*}
\end{small}
\begin{small}
\begin{eqnarray*}
&&\int_0^1\left(1-(2k-j+1)\frac{y}{k}\right)\left(1+\log\left(\frac{x}{k}+\frac{y}{k}-(2k-j+1)\frac{xy}{k^2}\right)\right)\mu(dy)\\
&=&\left(1-(2k-j+1)\frac{y}{k}\right)\left(1+\log\left(\frac{x}{k}+\frac{y}{k}-(2k-j+1)\frac{xy}{k^2}\right)\right)\gamma(y)\Bigg|_{y=0}^{y=1}-\int_0^1\gamma(y)\frac{\partial}{\partial y}\left[\left(1-(2k-j+1)\frac{y}{k}\right)\left(1+\log\left(\frac{x}{k}+\frac{y}{k}-(2k-j+1)\frac{xy}{k^2}\right)\right)\right]dy\\
&=&\left(1-(2k-j+1)\frac{1}{k}\right)\left(1+\log\left(\frac{x}{k}+\frac{1}{k}-(2k-j+1)\frac{x}{k^2}\right)\right)\\
&&-\frac{1}{k}\int_0^1\gamma(y)\left[-(2k-j+1)\left(1+\log(\frac{x}{k}+\frac{y}{k}-(2k-j+1)\frac{xy}{k^2})\right)+\frac{(1-(2k-j+1)\frac{x}{k})(1-(2k-j+1)\frac{y}{k})}{\frac{x}{k}+\frac{y}{k}-(2k-j+1)\frac{xy}{k^2}}\right]dy\\
&=&\left(1-(2k-j+1)\frac{1}{k}\right)\left(1+\log\left(\frac{x}{k}+\frac{1}{k}-(2k-j+1)\frac{x}{k^2}\right)\right)\\
&&-\frac{1}{k}\int_0^1\gamma(y)\left[-(2k-j+1)\left(1+\log\left(\frac{1-(1-(2k-j+1)\frac{x}{k})(1-(2k-j+1)\frac{y}{k})}{(2k-j+1)}\right)\right)+(2k-j+1)\frac{(1-(2k-j+1)\frac{x}{k})(1-(2k-j+1)\frac{y}{k})}{1-(1-(2k-j+1)\frac{x}{k})(1-(2k-j+1)\frac{y}{k})}\right]dy\\
&\implies& \frac{1}{k}\sum_{j=1}^{k-1}\int_0^1\gamma(x)\int_0^1\left(1-(2k-j+1)\frac{y}{k}\right)\left(1+\log(\frac{x}{k}+\frac{y}{k}-(2k-j+1)\frac{xy}{k^2})\right)\mu(dy)dx\\
&=&\frac{1}{k}\sum_{j=1}^{k-1}\int_0^1\gamma(x)\left(1-(2k-j+1)\frac{1}{k}\right)\left(1+\log(\frac{x}{k}+\frac{1}{k}-(2k-j+1)\frac{x}{k^2})\right)dx\\
&&-\frac{1}{k^2}\sum_{j=1}^{k-1}\int_0^1\int_0^1\gamma(x)\gamma(y)\Bigg[-(2k-j+1)\left(1+\log\left(\frac{1-(1-(2k-j+1)\frac{x}{k})(1-(2k-j+1)\frac{y}{k})}{(2k-j+1)}\right)\right)+\\&&+(2k-j+1)\frac{(1-(2k-j+1)\frac{x}{k})(1-(2k-j+1)\frac{y}{k})}{1-(1-(2k-j+1)\frac{x}{k})(1-(2k-j+1)\frac{y}{k})}\Bigg]dxdy.
\end{eqnarray*}
\end{small}
The non-linear part of the $F_{k,m}$ is then
\begin{small}
\begin{eqnarray*}
&&-\frac{1}{k^2}\left(\frac{k(k-1)}{2}-1\right)\int_0^1\int_0^1\gamma(x)\gamma(y)\log\left(1-\frac{1}{\left(1+\left(\frac{k(k-1)}{2}-1\right)\frac{x}{k}\right)\left(1+\left(\frac{k(k-1)}{2}-1\right)\frac{y}{k}\right)}\right)dxdy\\
&&-\frac{1}{k^2}\left(\frac{k(k-1)}{2}-1\right)\int_0^1\int_0^1\gamma(x)\gamma(y)\frac{1}{1-\frac{1}{\left(1+\left(\frac{k(k-1)}{2}-1\right)\frac{x}{k}\right)\left(1+\left(\frac{k(k-1)}{2}-1\right)\frac{y}{k}\right)}}dxdy\\
&&+\frac{1}{k^2}\sum_{j=1}^{k-1}(2k-j+1)\int_0^1\int_0^1\gamma(x)\gamma(y)\left[\log\left(1-(1-(2k-j+1)\frac{x}{k})(1-(2k-j+1)\frac{y}{k})\right)-\frac{1}{1-(1-(2k-j+1)\frac{x}{k})(1-(2k-j+1)\frac{y}{k})}\right]dxdy.
\end{eqnarray*}
\end{small}
This is of the form
\begin{small}
\[-\sum_{i=0}^{\infty}\left(\int_0^1\gamma(x)a_{i,k}\left(x\right)dx\right)^2\]
\end{small}
for some functions $a_{i,k}(x)$ in $x$. Therefore, $F_{k,m}$ is concave, as required.
\end{proof}
Since $F_{k,m}$ is concave on the compact convex space $\cal{M}_{\phi}$, Dubins' theorem~\cite{Dubins} implies that the minimizers of $F_{k,m}$ are measures supported on at most two points. This is because we only have one linear condition $\mathop{\mathbb{E}}_{x\sim\mu}[x]=\phi$ in defining $\cal{M}_{\phi}$. Furthermore, we have the following structural improvement for the minimizers of $F_{k,m}$.
\begin{lemma}\label{measuretypes}Suppose $\phi<\frac{1}{2}$ and $k\geq 3$. If $\mu$ is a minimizer of $F_{k,m}:\cal{M}_{\phi}\rightarrow\mathbb{R}$ that is supported on exactly two points, then either $0$ or $1$ is in the support of $\mu$.
\end{lemma}
\begin{proof}Suppose first that $\mu=p\delta_x+(1-p)\delta_y$, where $0<p<1$ and $0<y<x<1$, is a probability measure that is a minimizer of $F$. Then, given $z\in\left[0,1\right]$, we can write $z=\lambda x+(1-\lambda)y$, $\lambda\in\mathbb{R}$. Consider the measure
\begin{small}
\[\mu_t:=\mu+t(\delta_z-\lambda\delta_x-(1-\lambda)\delta_y).\]
\end{small}
Since $0<p<1$, for small enough $t$, $\mu_t$ is a probability measure. Furthermore, $\mu_t\in\cal{M}_{\phi}$ for small enough $t>0$. Since $\mu$ is a minimizer,
\begin{small}
\[\lim_{t\rightarrow 0^+}\frac{F_{k,m}(\mu_t)-F_{k,m}(\mu)}{t}\geq 0.\]
\end{small}
Explicitly, this is equivalent to the condition that
\begin{small}
\begin{equation}\label{fkineq}f_{k,m}(z)-\lambda f_{k,m}(x)-(1-\lambda)f_{k,m}(y)\geq 0,\end{equation}
\end{small}
where
\begin{small}
\begin{eqnarray*}f_{k,m}(s)&:=&2\left(pg_{k,m}\left(\frac{x}{k},\frac{s}{k}\right)+(1-p)g_{k,m}\left(\frac{y}{k},\frac{s}{k}\right)\right)-h_{k,m}\left(\frac{s}{k}\right)\\
&=&-2p\left((1-x)(1-s)\log((1-x)(1-s))+\left(\frac{x}{k}+\frac{s}{k}+\left(\frac{k(k-1)}{2}-1\right)\frac{xs}{k^2}\right)\log\left(\frac{x}{k}+\frac{s}{k}+\left(\frac{k(k-1)}{2}-1\right)\frac{xs}{k^2}\right)\right)\\
&&-2p\sum_{j=1}^{k-1}\left(\frac{x}{k}+\frac{s}{k}-(2k-j+1)\frac{xs}{k^2}\right)\log\left(\frac{x}{k}+\frac{s}{k}-(2k-j+1)\frac{xs}{k^2}\right)\\
&&-2(1-p)\left((1-y)(1-s)\log((1-y)(1-s))+\left(\frac{y}{k}+\frac{s}{k}+\left(\frac{k(k-1)}{2}-1\right)\frac{ys}{k^2}\right)\log\left(\frac{y}{k}+\frac{s}{k}+\left(\frac{k(k-1)}{2}-1\right)\frac{ys}{k^2}\right)\right)\\
&&-2(1-p)\sum_{j=1}^{k-1}\left(\frac{y}{k}+\frac{s}{k}-(2k-j+1)\frac{ys}{k^2}\right)\log\left(\frac{y}{k}+\frac{s}{k}-(2k-j+1)\frac{ys}{k^2}\right)\\
&&+s\log \frac{s}{k}+(1-s)\log(1-s)+(1-2\phi)(1-s)\log m.
\end{eqnarray*}
\end{small}
$f_{k,m}$ is smooth within $\left(0,1\right)$. By inequality~\eqref{fkineq}, since $z\in\left[0,1\right]$ was arbitrary, the function $f_{k,m}$ lies above the line connecting $(x,f_{k,m}(x))$ and $(y,f_{k,m}(y))$. From these, it follows that a) $f_{k,m}'(x)=f_{k,m}'(y)=\frac{f_{k,m}(x)-f_{k,m}(y)}{x-y}$ and b) $f_{k,m}''(x),f_{k,m}''(y)\geq 0$. By Rolle's theorem, a) implies that $f_{k,m}''(z)\leq 0$ for some $z\in[y,x]$. We show these are not possible with our specific function $f_{k,m}$. Indeed, we have
\begin{small}
\begin{eqnarray*}
f_{k,m}'(s)&=&-2p\left(-(1-x)\left(1+\log((1-x)(1-s))\right)+\left(\frac{1}{k}+\left(\frac{k(k-1)}{2}-1\right)\frac{x}{k^2}\right)\left(1+\log\left(\frac{x}{k}+\frac{s}{k}+\left(\frac{k(k-1)}{2}-1\right)\frac{xs}{k^2}\right)\right)\right)\\
&&-2p\sum_{j=1}^{k-1}\left(\frac{1}{k}-(2k-j+1)\frac{x}{k^2}\right)\left(1+\log\left(\frac{x}{k}+\frac{s}{k}-(2k-j+1)\frac{xs}{k^2}\right)\right)\\
&&-2(1-p)\left(-(1-y)\left(1+\log((1-y)(1-s))\right)+\left(\frac{1}{k}+\left(\frac{k(k-1)}{2}-1\right)\frac{y}{k^2}\right)\left(1+\log\left(\frac{y}{k}+\frac{s}{k}+\left(\frac{k(k-1)}{2}-1\right)\frac{ys}{k^2}\right)\right)\right)\\
&&-2(1-p)\sum_{j=1}^{k-1}\left(\frac{1}{k}-(2k-j+1)\frac{y}{k^2}\right)\left(1+\log\left(\frac{y}{k}+\frac{s}{k}-(2k-j+1)\frac{ys}{k^2}\right)\right)\\
&&+\log\frac{s}{1-s}-\log k-(1-2\phi)\log m
\end{eqnarray*}
\end{small}
and
\begin{small}
\begin{eqnarray*}
f_{k,m}''(s)&=&2p\frac{1-x}{s-1}-2p\frac{\left(\frac{1}{k}+\left(\frac{k(k-1)}{2}-1\right)\frac{x}{k^2}\right)^2}{\frac{x}{k}+\frac{s}{k}+\left(\frac{k(k-1)}{2}-1\right)\frac{xs}{k^2}}-2p\sum_{j=1}^{k-1}\frac{\left(\frac{1}{k}-(2k-j+1)\frac{x}{k^2}\right)^2}{\frac{x}{k}+\frac{s}{k}-(2k-j+1)\frac{xs}{k^2}}+2(1-p)\frac{1-y}{s-1}-2(1-p)\frac{\left(\frac{1}{k}+\left(\frac{k(k-1)}{2}-1\right)\frac{y}{k^2}\right)^2}{\frac{y}{k}+\frac{s}{k}+\left(\frac{k(k-1)}{2}-1\right)\frac{ys}{k^2}}\\
&&-2(1-p)\sum_{j=1}^{k-1}\frac{\left(\frac{1}{k}-(2k-j+1)\frac{y}{k^2}\right)^2}{\frac{y}{k}+\frac{s}{k}-(2k-j+1)\frac{ys}{k^2}}-\frac{1}{s(s-1)}\\
&=&2\frac{1-\phi}{s-1}-2p\frac{\frac{1}{k}+\left(\frac{k(k-1)}{2}-1\right)\frac{x}{k^2}}{s+\frac{x}{1+\left(\frac{k(k-1)}{2}-1\right)\frac{x}{k}}}-2p\sum_{j=1}^{k-1}\frac{\frac{1}{k}-(2k-j+1)\frac{x}{k^2}}{s+\frac{x}{1-(2k-j+1)\frac{x}{k}}}-2(1-p)\frac{\frac{1}{k}+\left(\frac{k(k-1)}{2}-1\right)\frac{y}{k^2}}{s+\frac{y}{1+\left(\frac{k(k-1)}{2}-1\right)\frac{y}{k}}}-2(1-p)\sum_{j=1}^{k-1}\frac{\frac{1}{k}-(2k-j+1)\frac{y}{k^2}}{s+\frac{y}{1-(2k-j+1)\frac{y}{k}}}-\frac{1}{s(s-1)}.
\end{eqnarray*}
\end{small}
We show that $sf_{k,m}''(s)$ is a decreasing function. Multiplication by a function, in this case the identity function, is a trick we also use later in the paper. Note that
\begin{small}
\begin{eqnarray*}
f_{k,m}'''(s)&=&-\frac{1-2\phi}{(s-1)^2}+2p\frac{\frac{1}{k}+\left(\frac{k(k-1)}{2}-1\right)\frac{x}{k^2}}{\left(s+\frac{x}{1+\left(\frac{k(k-1)}{2}-1\right)\frac{x}{k}}\right)^2}+2p\frac{1}{k}\sum_{j=1}^{k-1}\frac{1-(2k-j+1)\frac{x}{k}}{\left(s+\frac{x}{1-(2k-j+1)\frac{x}{k}}\right)^2}\\&&+2(1-p)\frac{\frac{1}{k}+\left(\frac{k(k-1)}{2}-1\right)\frac{y}{k^2}}{\left(s+\frac{y}{1+\left(\frac{k(k-1)}{2}-1\right)\frac{y}{k}}\right)^2}+2(1-p)\frac{1}{k}\sum_{j=1}^{k-1}\frac{1-(2k-j+1)\frac{y}{k}}{\left(s+\frac{y}{1-(2k-j+1)\frac{y}{k}}\right)^2}-\frac{1}{s^2}
\end{eqnarray*}
\end{small}
It follows that
\begin{small}
\begin{eqnarray*}
&&\frac{d}{ds}(sf_{k,m}''(s))\\&=&2\frac{1-\phi}{s-1}-2p\frac{\frac{1}{k}+\left(\frac{k(k-1)}{2}-1\right)\frac{x}{k^2}}{s+\frac{x}{1+\left(\frac{k(k-1)}{2}-1\right)\frac{x}{k}}}-2p\sum_{j=1}^{k-1}\frac{\frac{1}{k}-(2k-j+1)\frac{x}{k^2}}{s+\frac{x}{1-(2k-j+1)\frac{x}{k}}}-2(1-p)\frac{\frac{1}{k}+\left(\frac{k(k-1)}{2}-1\right)\frac{y}{k^2}}{s+\frac{y}{1+\left(\frac{k(k-1)}{2}-1\right)\frac{y}{k}}}-2(1-p)\sum_{j=1}^{k-1}\frac{\frac{1}{k}-(2k-j+1)\frac{y}{k^2}}{s+\frac{y}{1-(2k-j+1)\frac{y}{k}}}-\frac{1}{s(s-1)}\\
&&-\frac{s(1-2\phi)}{(s-1)^2}+2ps\frac{\frac{1}{k}+\left(\frac{k(k-1)}{2}-1\right)\frac{x}{k^2}}{\left(s+\frac{x}{1+\left(\frac{k(k-1)}{2}-1\right)\frac{x}{k}}\right)^2}+2ps\frac{1}{k}\sum_{j=1}^{k-1}\frac{1-(2k-j+1)\frac{x}{k}}{\left(s+\frac{x}{1-(2k-j+1)\frac{x}{k}}\right)^2}\\&&+2(1-p)s\frac{\frac{1}{k}+\left(\frac{k(k-1)}{2}-1\right)\frac{y}{k^2}}{\left(s+\frac{y}{1+\left(\frac{k(k-1)}{2}-1\right)\frac{y}{k}}\right)^2}+2(1-p)s\frac{1}{k}\sum_{j=1}^{k-1}\frac{1-(2k-j+1)\frac{y}{k}}{\left(s+\frac{y}{1-(2k-j+1)\frac{y}{k}}\right)^2}-\frac{1}{s}\\
&=&-\frac{1-2\phi}{(1-s)^2}-2p\frac{x}{\left(s+\frac{x}{1+\left(\frac{k(k-1)}{2}-1\right)\frac{x}{k}}\right)^2}-2p\frac{1}{k}\sum_{j=1}^{k-1}\frac{x}{\left(s+\frac{x}{1-(2k-j+1)\frac{x}{k}}\right)^2}-2(1-p)\frac{y}{\left(s+\frac{y}{1+\left(\frac{k(k-1)}{2}-1\right)\frac{y}{k}}\right)^2}-2(1-p)\frac{1}{k}\sum_{j=1}^{k-1}\frac{y}{\left(s+\frac{y}{1-(2k-j+1)\frac{y}{k}}\right)^2}\\
&\leq &-\frac{1-2\phi}{(1-s)^2}.
\end{eqnarray*}
\end{small}
Therefore, $sf_{k,m}''(s)$ is a strictly decreasing function when $\phi<\frac{1}{2}$. This contradicts the existence of $z\in(y,x)$ such that $f_{k,m}''(z)\leq 0$ since $f_{k,m}''(x),f_{k,m}''(y)\geq 0$.
\end{proof}
Therefore, in order to bound $F_{k,m}(\mu)$ from below, it suffices to consider three types of measures:
\begin{enumerate}[I)]
\item $\mu=\delta_{\phi}$;
\item $\mu=(1-p)\delta_0+p\delta_x$, where $0<p<1$ and $px=\phi$; and
\item $\mu=(1-p)\delta_1+p\delta_x$, where $0<p<1$ and $p=\frac{1-\phi}{1-x}$.
\end{enumerate}
\vspace{3mm}
We address each of these in the next three sections.
\section{Type I optimization}
In the case of type I measures $\delta_{\phi}$, we have
\begin{small}
\begin{eqnarray*}&&F_{k,m}(\delta_{\phi})\\
&=&g_{k,m}\left(\frac{\phi}{k},\frac{\phi}{k}\right)-h_{k,m}\left(\frac{\phi}{k}\right)\\
&=&-(1-\phi)^2\log((1-\phi)^2)-\left(\frac{2\phi}{k}+\left(\frac{k(k-1)}{2}-1\right)\frac{\phi^2}{k^2}\right)\log\left(\frac{2\phi}{k}+\left(\frac{k(k-1)}{2}-1\right)\frac{\phi^2}{k^2}\right)-\sum_{j=1}^{k-1}\left(\frac{2\phi}{k}-(2k-j+1)\frac{\phi^2}{k^2}\right)\log\left(\frac{2\phi}{k}-(2k-j+1)\frac{\phi^2}{k^2}\right)\\
&&+\phi\log\frac{\phi}{k}+(1-\phi)\log(1-\phi)-\phi(1-\phi)\log m
\end{eqnarray*}
\end{small}

\begin{proposition}For integers $k\geq 5$ and $1\leq m\leq \sqrt{k}$, and for every $0<\phi<\frac{1}{2}$, we have $F_{k,m}(\delta_{\phi})>0$.
\end{proposition}
\begin{proof}
First, note that $F_{k,m}(\delta_0)=0$ and for every $k\geq 3$, and 

\begin{small}
\begin{eqnarray*}
&&F_{k,m}\left(\delta_{\frac{1}{2}}\right)\\
&=&-\frac{1}{4}\log\frac{1}{4}-\left(\frac{1}{k}+\left(\frac{k(k-1)}{2}-1\right)\frac{1}{4k^2}\right)\log\left(\frac{1}{k}+\left(\frac{k(k-1)}{2}-1\right)\frac{1}{4k^2}\right)-\sum_{j=1}^{k-1}\left(\frac{1}{k}-(2k-j+1)\frac{1}{4k^2}\right)\log\left(\frac{1}{k}-(2k-j+1)\frac{1}{4k^2}\right)\\
&&+\frac{1}{2}\log\frac{1}{2k}+\frac{1}{2}\log\frac{1}{2}-\frac{1}{4}\log m\\
&=&-\left(\frac{k^2+7k-2}{8k^2}\right)\log\left(\frac{k^2+7k-2}{8k^2}\right)-\frac{1}{2}\log(2k)-\sum_{j=1}^{k-1}\left(\frac{2k+j-1}{4k^2}\right)\log\left(\frac{2k+j-1}{4k^2}\right)-\frac{1}{4}\log m\\
&=&-\left(\frac{k^2+7k-2}{8k^2}\right)\log\left(\frac{k^2+7k-2}{8k^2}\right)-\frac{1}{2}\log(2k)-\frac{1}{k}\sum_{j=0}^{k-1}\left(\frac{1}{2}+\frac{j}{4k}\right)\log\left(\frac{1}{2}+\frac{j}{4k}\right)+(\log k)\sum_{j=0}^{k-2}\left(\frac{2k+j}{4k^2}\right)\\&&+\frac{1}{k}\left(\frac{1}{2}+\frac{k-1}{4k}\right)\log\left(\frac{1}{2}+\frac{k-1}{4k}\right)-\frac{1}{4}\log m\\
&\geq& -\left(\frac{k^2+7k-2}{8k^2}\right)\log\left(\frac{k^2+7k-2}{8k^2}\right)-\frac{1}{2}\log(2k)-\int_0^1\left(\frac{1}{2}+\frac{t}{4}\right)\log\left(\frac{1}{2}+\frac{t}{4}\right)dt+(\log k)\frac{(k-1)(5k-2)}{8k^2}\\&&+\frac{1}{k}\left(\frac{1}{2}+\frac{k-1}{4k}\right)\log\left(\frac{1}{2}+\frac{k-1}{4k}\right)-\frac{1}{4}\log \lfloor\sqrt{k}\rfloor.
\end{eqnarray*}
\end{small}

The latter quantity is positive for $k=5,6$. We show that
\begin{small}
\begin{eqnarray*}
\frac{d}{dk}\left(-\left(\frac{k^2+7k-2}{8k^2}\right)\log\left(\frac{k^2+7k-2}{8k^2}\right)-\frac{1}{2}\log(2k)-\int_0^1\left(\frac{1}{2}+\frac{x}{4}\right)\log\left(\frac{1}{2}+\frac{x}{4}\right)dx+(\log k)\frac{(k-1)(5k-2)}{8k^2}+\frac{1}{k}\left(\frac{1}{2}+\frac{k-1}{4k}\right)\log\left(\frac{1}{2}+\frac{k-1}{4k}\right)-\frac{\log k}{8}\right)\geq 0,
\end{eqnarray*}
\end{small}
for $k\geq 6$, from which it would follows that $F_{k,m}\left(\delta_{\frac{1}{2}}\right)>0$ for integers $k\geq 5$ and $1\leq m\leq \sqrt{k}$. Indeed, the derivative is equal to
\begin{small}
\begin{eqnarray*}
\frac{(7k-4)\log\left(\frac{k^2+7k-2}{8k}\right)+(4-6k)\log\left(\frac{3}{4}-\frac{1}{4k}\right)}{8k^3}.
\end{eqnarray*}
\end{small}

This is positive for $k\geq 1$. Therefore, to show that $F_{k,m}\left(\delta_{\phi}\right)>0$ for $0<\phi<\frac{1}{2}$, integers $k\geq 5$ and $1\leq m\leq \sqrt{k}$, it suffices to show that $F_{k,m}\left(\delta_{\phi}\right)$ is strictly concave in for such $\phi$, given such $k,m$. We have
\begin{small}
\begin{eqnarray*}
&&\frac{d^2}{d\phi^2}F_{k,m}\left(\delta_{\phi}\right)\\
&=&\frac{1}{1-\phi}+\frac{1}{\phi}-4\log(1-\phi)-6-\frac{2}{k^2}\left(\frac{k(k-1)}{2}-1\right)-\frac{2}{k^2}\left(\frac{k(k-1)}{2}-1\right)\log\left(\frac{2\phi}{k}+\left(\frac{k(k-1)}{2}-1\right)\frac{\phi^2}{k^2}\right)-\frac{\left(\frac{2}{k}+2\left(\frac{k(k-1)}{2}-1\right)\frac{\phi}{k^2}\right)^2}{\frac{2\phi}{k}+\left(\frac{k(k-1)}{2}-1\right)\frac{\phi^2}{k^2}}\\
&&+\sum_{j=1}^{k-1}\frac{2(2k-j+1)}{k^2}\left(1+\log\left(\frac{2\phi}{k}-(2k-j+1)\frac{\phi^2}{k^2}\right)\right)-\sum_{j=1}^{k-1}\frac{\left(\frac{2}{k}-2(2k-j+1)\frac{\phi}{k^2}\right)^2}{\frac{2\phi}{k}-(2k-j+1)\frac{\phi^2}{k^2}}+2\log m\\
&\leq& \frac{1}{1-\phi}+\frac{1}{\phi}-4\log(1-\phi)-6-\frac{2}{k^2}\left(\frac{k(k-1)}{2}-1\right)-\frac{2}{k^2}\left(\frac{k(k-1)}{2}-1\right)\log\left(\frac{2\phi}{k}+\left(\frac{k(k-1)}{2}-1\right)\frac{\phi^2}{k^2}\right)-\frac{\left(\frac{2}{k}+2\left(\frac{k(k-1)}{2}-1\right)\frac{\phi}{k^2}\right)^2}{\frac{2\phi}{k}+\left(\frac{k(k-1)}{2}-1\right)\frac{\phi^2}{k^2}}\\&&+\sum_{j=1}^{k-1}\frac{2(2k-j+1)}{k^2}\left(1+\log\left(\frac{2\phi}{k}-(2k-j+1)\frac{\phi^2}{k^2}\right)\right)+\log k\\
&=& \frac{1}{1-\phi}+\frac{1}{\phi}-4\log(1-\phi)-6-\frac{2}{k^2}\left(\frac{k(k-1)}{2}-1\right)-\frac{2}{k^2}\left(\frac{k(k-1)}{2}-1\right)\log\left(\frac{2\phi}{k}+\left(\frac{k(k-1)}{2}-1\right)\frac{\phi^2}{k^2}\right)-\frac{\left(\frac{2}{k}+2\left(\frac{k(k-1)}{2}-1\right)\frac{\phi}{k^2}\right)^2}{\frac{2\phi}{k}+\left(\frac{k(k-1)}{2}-1\right)\frac{\phi^2}{k^2}}\\&&+(1-\log k)\frac{3k^2-k-2}{k^2}+2\sum_{j=0}^{k-1}\frac{1}{k}\left(1+\frac{k-j}{k}\right)\log\left(2\phi-\phi^2-(k-j)\frac{\phi^2}{k}\right)-2\frac{1}{k}\left(1+\frac{1}{k}\right)\log\left(2\phi-\phi^2-\frac{\phi^2}{k}\right)+\log k.
\end{eqnarray*}
We may bound the latter quantity from above to obtain
\begin{eqnarray*}
&&\frac{d^2}{d\phi^2}F_{k,m}\left(\delta_{\phi}\right)\\
&\leq& \frac{1}{1-\phi}+\frac{1}{\phi}-4\log(1-\phi)-6-\frac{2}{k^2}\left(\frac{k(k-1)}{2}-1\right)-\frac{2}{k^2}\left(\frac{k(k-1)}{2}-1\right)\log\left(\frac{2\phi}{k}+\left(\frac{k(k-1)}{2}-1\right)\frac{\phi^2}{k^2}\right)-\frac{\left(\frac{2}{k}+2\left(\frac{k(k-1)}{2}-1\right)\frac{\phi}{k^2}\right)^2}{\frac{2\phi}{k}+\left(\frac{k(k-1)}{2}-1\right)\frac{\phi^2}{k^2}}\\&&+(1-\log k)\frac{3k^2-k-2}{k^2}+2\int_0^1(2-x)\log(2\phi-(2-x)\phi^2)dx-2\frac{1}{k}\left(1+\frac{1}{k}\right)\log\left(2\phi-\phi^2-\frac{\phi^2}{k}\right)+\log k\\
&=& \frac{1}{1-\phi}+\frac{1}{\phi}-4\log(1-\phi)-6-\frac{2}{k^2}\left(\frac{k(k-1)}{2}-1\right)-\frac{2}{k^2}\left(\frac{k(k-1)}{2}-1\right)\log\left(\frac{2\phi}{k}+\left(\frac{k(k-1)}{2}-1\right)\frac{\phi^2}{k^2}\right)-\frac{\left(\frac{2}{k}+2\left(\frac{k(k-1)}{2}-1\right)\frac{\phi}{k^2}\right)^2}{\frac{2\phi}{k}+\left(\frac{k(k-1)}{2}-1\right)\frac{\phi^2}{k^2}}\\&&+(1-\log k)\frac{3k^2-k-2}{k^2}-\frac{8(1-\phi^2)\log(2(1-\phi))+2(4-\phi^2)\log(2-\phi)+\phi(3\phi-6\phi\log\phi+4)}{4\phi^2}-2\frac{1}{k}\left(1+\frac{1}{k}\right)\log\left(2\phi-\phi^2-\frac{\phi^2}{k}\right)+\log k
\end{eqnarray*}
\end{small}

In order to show that this is negative for $k\geq 2$ and $0<\phi<\frac{1}{2}$, we first show that it is an increasing function of $\phi$. Indeed, it has partial derivative with respect to $\phi$ equal to
\begin{small}
\begin{eqnarray*}
&&\frac{1}{(1-\phi)^2}-\frac{1}{\phi^2}+\frac{4}{1-\phi}-\frac{2}{k^2}\left(\frac{k(k-1)}{2}-1\right)\frac{2+2\left(\frac{k(k-1)}{2}-1\right)\frac{\phi}{k}}{2\phi+\left(\frac{k(k-1)}{2}-1\right)\frac{\phi^2}{k}}+\frac{16(2k+\phi(k^2-k-2))}{\phi^2(4k+\phi(k^2-k-2))^2}\\&&+\frac{4(\phi+\phi^2+\log(2)+\log(1-\phi)+\log(2-\phi))}{\phi^3}-2\frac{1}{k}\left(1+\frac{1}{k}\right)\frac{2-2\phi-\frac{2\phi}{k}}{2\phi-\phi^2-\frac{\phi^2}{k}}\\
&>& \frac{1}{(1-\phi)^2}-\frac{1}{\phi^2}+\frac{4}{1-\phi}+\frac{4(\phi+\phi^2+\log(2)+\log(1-\phi)+\log(2-\phi))}{\phi^3}-\frac{2}{\phi}>0,
\end{eqnarray*}
\end{small}
where the last inequality is true for every $\phi>0$. Therefore,
\begin{small}
\begin{eqnarray*}
&&\frac{1}{1-\phi}+\frac{1}{\phi}-4\log(1-\phi)-6-\frac{2}{k^2}\left(\frac{k(k-1)}{2}-1\right)-\frac{2}{k^2}\left(\frac{k(k-1)}{2}-1\right)\log\left(\frac{2\phi}{k}+\left(\frac{k(k-1)}{2}-1\right)\frac{\phi^2}{k^2}\right)-\frac{\left(\frac{2}{k}+2\left(\frac{k(k-1)}{2}-1\right)\frac{\phi}{k^2}\right)^2}{\frac{2\phi}{k}+\left(\frac{k(k-1)}{2}-1\right)\frac{\phi^2}{k^2}}\\&&+(1-\log k)\frac{3k^2-k-2}{k^2}-\frac{8(1-\phi^2)\log(2(1-\phi))+2(4-\phi^2)\log(2-\phi)+\phi(3\phi-6\phi\log\phi+4)}{4\phi^2}-2\frac{1}{k}\left(1+\frac{1}{k}\right)\log\left(2\phi-\phi^2-\frac{\phi^2}{k}\right)
\end{eqnarray*}
\end{small}
is an increasing function of $\phi$. Since it is also negative for negative for $\phi=\frac{1}{2}$ when $k\geq 2$, it is negative for every $k\geq 2$ and every $0<\phi<\frac{1}{2}$. From the above calculations, $\frac{d^2}{d\phi^2}F_{k,m}\left(\delta_{\phi}\right)$ is bounded above by this quantity. This implies the concavity of $F_{k,m}\left(\delta_{\phi}\right)$ in $0\leq\phi\leq\frac{1}{2}$ for integers $k\geq 2$ and $1\leq m\leq\sqrt{k}$. The conclusion follows.
\end{proof}

\section{Type II optimization}
For type II measures $\left(1-\frac{\phi}{x}\right)\delta_0+\frac{\phi}{x}\delta_x$, we have
\begin{small}
\begin{eqnarray*}&&F_{k,m}\left(\left(1-\frac{\phi}{x}\right)\delta_0+\frac{\phi}{x}\delta_x\right)\\&=&2\frac{\phi}{x}\left(1-\frac{\phi}{x}\right)g_{k,m}\left(\frac{x}{k},0\right)+\left(\frac{\phi}{x}\right)^2g_{k,m}\left(\frac{x}{k},\frac{x}{k}\right)+\left(1-\frac{\phi}{x}\right)^2g_{k,m}(0,0)-\frac{\phi}{x}h_{k,m}\left(\frac{x}{k}\right)-\left(1-\frac{\phi}{x}\right)h_{k,m}(0)\\
&=&-\left(\frac{\phi}{x}\right)^2\left((1-x)^2\log((1-x)^2)+\left(2\frac{x}{k}+\left(\frac{k(k-1)}{2}-1\right)\frac{x^2}{k^2}\right)\log\left(2\frac{x}{k}+\left(\frac{k(k-1)}{2}-1\right)\frac{x^2}{k^2}\right)\right)\\
&&-\left(\frac{\phi}{x}\right)^2\sum_{j=1}^{k-1}\left(2\frac{x}{k}-(2k-j+1)\frac{x^2}{k^2}\right)\log\left(2\frac{x}{k}-(2k-j+1)\frac{x^2}{k^2}\right)-\left(\frac{\phi}{x}\right)\left(1-2\frac{\phi}{x}\right)\left(x\log\left(\frac{x}{k}\right)+(1-x)\log(1-x)\right)-\phi(1-\phi)\log m
\end{eqnarray*}
\end{small}

\begin{proposition}For integers $k\geq 5$ and $1\leq m\leq \sqrt{k}$, and for every $0<\phi<\frac{1}{2}$ and $x\in[\phi,1]$, we have
\[F_{k,m}\left(\left(1-\frac{\phi}{x}\right)\delta_0+\frac{\phi}{x}\delta_x\right)>0.\]
\end{proposition}
\begin{proof}First, note that when $x=\phi$, we have $\left(1-\frac{\phi}{x}\right)\delta_0+\frac{\phi}{x}\delta_x=\delta_{\phi}$ reducing to type I measures addressed in the previous section. When $x=1$, we have
\begin{small}
\begin{eqnarray*}
&&F_{k,m}\left(\left(1-\phi\right)\delta_0+\phi\delta_1\right)\\&=&-\phi^2\left(\left(2\frac{1}{k}+\left(\frac{k(k-1)}{2}-1\right)\frac{1}{k^2}\right)\log\left(2\frac{1}{k}+\left(\frac{k(k-1)}{2}-1\right)\frac{1}{k^2}\right)+\sum_{j=1}^{k-1}\left(2\frac{1}{k}-(2k-j+1)\frac{1}{k^2}\right)\log\left(2\frac{1}{k}-(2k-j+1)\frac{1}{k^2}\right)\right)\\&&+\phi\left(1-2\phi\right)\log k-\phi(1-\phi)\log m\\
&=& -\phi^2\left(\left(\frac{k^2+3k-2}{2k^2}\right)\log\left(\frac{k^2+3k-2}{2k^2}\right)-\left(\frac{k-1}{k^2}\right)\log\left(\frac{k-1}{k^2}\right)-\sum_{j=0}^{k-1}\frac{j\log k}{k^2}+\sum_{j=0}^{k-1}\frac{1}{k}\cdot\frac{j}{k}\log\left(\frac{j}{k}\right)\right)+\phi\left(1-2\phi\right)\log k-\phi(1-\phi)\log m\\
&\geq& -\phi^2\left(\left(\frac{k^2+3k-2}{2k^2}\right)\log\left(\frac{k^2+3k-2}{2k^2}\right)-\left(\frac{k-1}{k^2}\right)\log\left(\frac{k-1}{k^2}\right)-\frac{(k-1)\log k}{2k}+\sum_{j=0}^{k-1}\frac{1}{k}\cdot\frac{j}{k}\log\left(\frac{j}{k}\right)\right)+\phi\left(1-2\phi\right)\log k-\phi(1-\phi)\log \lfloor\sqrt{k}\rfloor\\
&=& -\phi\left(\phi\left(\left(\frac{k^2+3k-2}{2k^2}\right)\log\left(\frac{k^2+3k-2}{2k^2}\right)-\left(\frac{k-1}{k^2}\right)\log\left(\frac{k-1}{k^2}\right)-\frac{(k-1)\log k}{2k}+\sum_{j=0}^{k-1}\frac{1}{k}\cdot\frac{j}{k}\log\left(\frac{j}{k}\right)+\frac{3}{2}\log k\right)-\frac{1}{2}\log k\right)+\phi(1-\phi)\log\left(\frac{\sqrt{k}}{\lfloor\sqrt{k}\rfloor}\right).
\end{eqnarray*}
\end{small}
We claim that this is positive for $k\geq 5$ and $0<\phi<\frac{1}{2}$. We first show it for $k\geq 13$. Since the first term is a quadratic in $\phi$ with $\phi=0$ as a root, this would follow if it is concave and the other root is to the right of $\frac{1}{2}$. This would follow from having
\begin{small}
\begin{equation}\label{quadraticineq}0<\left(\frac{k^2+3k-2}{2k^2}\right)\log\left(\frac{k^2+3k-2}{2k^2}\right)-\left(\frac{k-1}{k^2}\right)\log\left(\frac{k-1}{k^2}\right)-\frac{(k-1)\log k}{2k}+\sum_{j=0}^{k-1}\frac{1}{k}\cdot\frac{j}{k}\log\left(\frac{j}{k}\right)+\frac{3}{2}\log k< \log k
\end{equation}
\end{small}
for $k\geq 6$. Indeed, using
\begin{small}
\[\sum_{j=0}^{k-1}\frac{1}{k}\cdot\frac{j}{k}\log\left(\frac{j}{k}\right)\geq -\frac{(k-1)\log k}{2k},\]
\end{small}
we obtain
\begin{small}
\begin{eqnarray*}
&&\left(\frac{k^2+3k-2}{2k^2}\right)\log\left(\frac{k^2+3k-2}{2k^2}\right)-\left(\frac{k-1}{k^2}\right)\log\left(\frac{k-1}{k^2}\right)-\frac{(k-1)\log k}{2k}+\sum_{j=0}^{k-1}\frac{1}{k}\cdot\frac{j}{k}\log\left(\frac{j}{k}\right)+\frac{3}{2}\log k\\
&\geq& \left(\frac{k^2+3k-2}{2k^2}\right)\log\left(\frac{k^2+3k-2}{2k^2}\right)-\left(\frac{k-1}{k^2}\right)\log\left(\frac{k-1}{k^2}\right)-\frac{(k-1)\log k}{k}+\frac{3}{2}\log k.
\end{eqnarray*}
\end{small}
This is positive for $k\geq 2$, and so the left inequality in~\eqref{quadraticineq} is true. For the right hand inequality in~\eqref{quadraticineq}, note that
\begin{small}
\begin{eqnarray*}
&&\left(\frac{k^2+3k-2}{2k^2}\right)\log\left(\frac{k^2+3k-2}{2k^2}\right)-\left(\frac{k-1}{k^2}\right)\log\left(\frac{k-1}{k^2}\right)-\frac{(k-1)\log k}{2k}+\sum_{j=0}^{k-1}\frac{1}{k}\cdot\frac{j}{k}\log\left(\frac{j}{k}\right)+\frac{3}{2}\log k-\log k\\
&\leq& \left(\frac{k^2+3k-2}{2k^2}\right)\log\left(\frac{k^2+3k-2}{2k^2}\right)-\left(\frac{k-1}{k^2}\right)\log\left(\frac{k-1}{k^2}\right)-\frac{(k-1)\log k}{2k}+\frac{1}{2}\log k.
\end{eqnarray*}
\end{small}
This is negative for every $k\geq 13$. Furthermore, we can check that for $6\leq k\leq 12$, we also have that
\begin{small}
\begin{eqnarray*}
\left(\frac{k^2+3k-2}{2k^2}\right)\log\left(\frac{k^2+3k-2}{2k^2}\right)-\left(\frac{k-1}{k^2}\right)\log\left(\frac{k-1}{k^2}\right)-\frac{(k-1)\log k}{2k}+\sum_{j=0}^{k-1}\frac{1}{k}\cdot\frac{j}{k}\log\left(\frac{j}{k}\right)+\frac{3}{2}\log k-\log k< 0.
\end{eqnarray*}
\end{small}
We obtain that the inequalities in~\eqref{quadraticineq} are true for $k\geq 6$. We also easily check that
\begin{small}
\[F_{k,m}\left(\left(1-\phi\right)\delta_0+\phi\delta_1\right)\geq F_{k,\lfloor\sqrt{k}\rfloor}\left(\left(1-\phi\right)\delta_0+\phi\delta_1\right)>0\]
\end{small}
for $k=5$, $1\leq m\leq \sqrt{5}$, and $0<\phi<\frac{1}{2}$. Given that 
\begin{small}
\[F_{k,m}\left(\left(1-\frac{\phi}{x}\right)\delta_0+\frac{\phi}{x}\delta_x\right)>0\]
\end{small}
for integers $k\geq 5$ and $1\leq m\leq\sqrt{k}$ when $x\in\{\phi,1\}$ and $0<\phi<\frac{1}{2}$, the Lemma would be proved if we show that
\begin{small}
\[\tau(x):=xF_{k,m}\left(\left(1-\frac{\phi}{x}\right)\delta_0+\frac{\phi}{x}\delta_x\right)\]
\end{small}
is concave in $x\in[\phi,1]$ for such $k,m,\phi$. We calculate $\tau''(x)$.
\begin{small}
\begin{eqnarray*}
&&\tau'(x)\\
&=&\left(\frac{\phi}{x}\right)^2\left((1-x)^2\log((1-x)^2)+\left(2\frac{x}{k}+\left(\frac{k(k-1)}{2}-1\right)\frac{x^2}{k^2}\right)\log\left(2\frac{x}{k}+\left(\frac{k(k-1)}{2}-1\right)\frac{x^2}{k^2}\right)+\sum_{j=1}^{k-1}(2k-j+1)\frac{x^2}{k^2}\log\left(2\frac{x}{k}-(2k-j+1)\frac{x^2}{k^2}\right)\right)\\
&&-x\left(\frac{\phi}{x}\right)^2\left(-2(1-x)(1+\log((1-x)^2))+\left(\frac{2}{k}+2\left(\frac{k(k-1)}{2}-1\right)\frac{x}{k^2}\right)\left(1+\log\left(2\frac{x}{k}+\left(\frac{k(k-1)}{2}-1\right)\frac{x^2}{k^2}\right)\right)\right)\\
&&-x\left(\frac{\phi}{x}\right)^2\sum_{j=1}^{k-1}\left(\frac{2}{k}-2(2k-j+1)\frac{x}{k^2}\right)-2\left(\frac{\phi}{x}\right)^2\left(x\log\left(\frac{x}{k}\right)+(1-x)\log(1-x)\right)-\phi\left(1-2\frac{\phi}{x}\right)\log\left(\frac{x}{k(1-x)}\right)-\phi(1-\phi)\log m,
\end{eqnarray*}
\end{small}
from which we obtain
\begin{small}
\begin{eqnarray*}
&&\tau''(x)\\
&=&\left(\frac{\phi}{x}\right)^2\left(-4+\left(\frac{4}{k}+\frac{2x}{k^2}\left(\frac{k(k-1)}{2}-1\right)\right)\right)+\left(\frac{\phi}{x}\right)^2\left(-2x-x\frac{\left(\frac{2}{k}+2\left(\frac{k(k-1)}{2}-1\right)\frac{x}{k^2}\right)^2}{2\frac{x}{k}+\left(\frac{k(k-1)}{2}-1\right)\frac{x^2}{k^2}}\right)\\
&+&\left(\frac{\phi}{x}\right)^2\left(\frac{x(3k^2-k-2)}{k^2}+\frac{2(k-1)}{k}-4\sum_{j=0}^{k-1}\frac{1}{k}\cdot\frac{1}{2-\left(2-\frac{j}{k}\right)x}+\frac{4}{2k-(k+1)x}\right)+\left(\frac{\phi}{x}\right)^2\frac{2(k-1)}{k}-\frac{\phi}{x(1-x)}\left(1-2\frac{\phi}{x}\right)\\
&\leq& \left(\frac{\phi}{x}\right)^2\left(-4+\left(\frac{4}{k}+\frac{2x}{k^2}\left(\frac{k(k-1)}{2}-1\right)\right)\right)+\left(\frac{\phi}{x}\right)^2\left(-2x-x\frac{\left(\frac{2}{k}+2\left(\frac{k(k-1)}{2}-1\right)\frac{x}{k^2}\right)^2}{2\frac{x}{k}+\left(\frac{k(k-1)}{2}-1\right)\frac{x^2}{k^2}}\right)\\
&&+\left(\frac{\phi}{x}\right)^2\left(\frac{x(3k^2-k-2)}{k^2}+\frac{2(k-1)}{k}-\frac{4}{x}\log\left(\frac{2-x}{2(1-x)}\right)+\frac{4}{2k-(k+1)x}\right)+\left(\frac{\phi}{x}\right)^2\frac{2(k-1)}{k}-\frac{\phi}{x(1-x)}\left(1-2\frac{\phi}{x}\right)\\
&=&\left(\frac{\phi}{x}\right)^2\left(\frac{2x}{k^2}\left(\frac{k(k-1)}{2}-1\right)+\left(-2x-x\frac{\left(\frac{2}{k}+2\left(\frac{k(k-1)}{2}-1\right)\frac{x}{k^2}\right)^2}{2\frac{x}{k}+\left(\frac{k(k-1)}{2}-1\right)\frac{x^2}{k^2}}\right)+\frac{x(3k^2-k-2)}{k^2}-\frac{4}{x}\log\left(\frac{2-x}{2(1-x)}\right)+\frac{4}{2k-(k+1)x}\right)\\
&&-\frac{\phi}{x(1-x)}\left(1-2\frac{\phi}{x}\right)\\
&\leq& \left(\frac{\phi}{x}\right)^2\left(2x+\left(-x\frac{\left(\frac{2}{k}+2\left(\frac{k(k-1)}{2}-1\right)\frac{x}{k^2}\right)^2}{2\frac{x}{k}+\left(\frac{k(k-1)}{2}-1\right)\frac{x^2}{k^2}}\right)-\frac{4}{x}\log\left(\frac{2-x}{2(1-x)}\right)+\frac{4}{2k-(k+1)x}-\frac{1}{1-x}\left(\frac{x}{\phi}-2\right)\right).
\end{eqnarray*}
\end{small}
Since $0<\phi<\frac{1}{2}$, we have $\frac{x}{\phi}> 2x$ and so
\begin{small}
\[\tau''(x)< \left(\frac{\phi}{x}\right)^2\left(2x+\left(-x\frac{\left(\frac{2}{k}+2\left(\frac{k(k-1)}{2}-1\right)\frac{x}{k^2}\right)^2}{2\frac{x}{k}+\left(\frac{k(k-1)}{2}-1\right)\frac{x^2}{k^2}}\right)-\frac{4}{x}\log\left(\frac{2-x}{2(1-x)}\right)+\frac{4}{2k-(k+1)x}+2\right).\]
\end{small}
For $k\geq 4$, this is decreasing in $k$ for $0\leq x\leq 1$. Therefore, for $0\leq x\leq 1$, we have
\begin{small}
\[\tau''(x)\leq \left(\frac{\phi}{x}\right)^2\left(2x-\frac{\left(\frac{1}{2}+\frac{5x}{8}\right)^2}{\frac{1}{2}+\frac{5x}{16}}-\frac{4}{x}\log\left(\frac{2-x}{2(1-x)}\right)+\frac{4}{8-5x}+2\right)\leq 0.\]
\end{small}
The latter inequality follows from the fact that
\begin{small}
\[2x-\frac{\left(\frac{1}{2}+\frac{5x}{8}\right)^2}{\frac{1}{2}+\frac{5x}{16}}-\frac{4}{x}\log\left(\frac{2-x}{2(1-x)}\right)+\frac{4}{8-5x}+2\]
\end{small}
is a decreasing function that vanishes at $x=0$. We conclude that $\tau''(x)< 0$, as required.
\end{proof}
\section{Type III optimization}
For type III measures $\left(\frac{\phi-x}{1-x}\right)\delta_1+\left(\frac{1-\phi}{1-x}\right)\delta_x$, we have
\begin{small}
\begin{eqnarray*}&&F_{k,m}\left(\left(\frac{\phi-x}{1-x}\right)\delta_1+\left(\frac{1-\phi}{1-x}\right)\delta_x\right)\\&=&2\left(\frac{\phi-x}{1-x}\right)\left(\frac{1-\phi}{1-x}\right)g_{k,m}\left(\frac{x}{k},\frac{1}{k}\right)+\left(\frac{1-\phi}{1-x}\right)^2g_{k,m}\left(\frac{x}{k},\frac{x}{k}\right)+\left(\frac{\phi-x}{1-x}\right)^2g_{k,m}\left(\frac{1}{k},\frac{1}{k}\right)-\left(\frac{1-\phi}{1-x}\right)h_{k,m}\left(\frac{x}{k}\right)-\left(\frac{\phi-x}{1-x}\right)h_{k,m}\left(\frac{1}{k}\right)\\
&=& -2\left(\frac{\phi-x}{1-x}\right)\left(\frac{1-\phi}{1-x}\right)\left(\left(\frac{x}{k}+\frac{1}{k}+\left(\frac{k(k-1)}{2}-1\right)\frac{x}{k^2}\right)\log\left(\frac{x}{k}+\frac{1}{k}+\left(\frac{k(k-1)}{2}-1\right)\frac{x}{k^2}\right)+(\log k)\frac{(k-1)(k(x-2)+2x)}{2k^2}\right)\\
&&-2\left(\frac{\phi-x}{1-x}\right)\left(\frac{1-\phi}{1-x}\right)\sum_{j=0}^{k-1}\frac{1}{k}\left(1-\left(1-\frac{j}{k}\right)x\right)\log\left(1-\left(1-\frac{j}{k}\right)x\right)+2\left(\frac{\phi-x}{1-x}\right)\left(\frac{1-\phi}{1-x}\right)\frac{1}{k}\left(1-\frac{x}{k}\right)\log\left(1-\frac{x}{k}\right)\\
&&-\left(\frac{1-\phi}{1-x}\right)^2\left((1-x)^2\log((1-x)^2)+\left(2\frac{x}{k}+\left(\frac{k(k-1)}{2}-1\right)\frac{x^2}{k^2}\right)\log\left(2\frac{x}{k}+\left(\frac{k(k-1)}{2}-1\right)\frac{x^2}{k^2}\right)+\left(\log k\right)\frac{x(k-1)(2x+k(3x-4))}{2k^2}\right)\\
&&-\left(\frac{1-\phi}{1-x}\right)^2\sum_{j=0}^{k-1}\frac{1}{k}\left(2x-\left(2-\frac{j}{k}\right)x^2\right)\log\left(2-\left(2-\frac{j}{k}\right)x\right)+(\log x)\left(\frac{1-\phi}{1-x}\right)^2\frac{x(x+k(3x-4))}{2k}\\&&+\frac{1}{k}\left(\frac{1-\phi}{1-x}\right)^2\left(2x-(k+1)\frac{x^2}{k}\right)\log\left(2x-(k+1)\frac{x^2}{k}\right)\\&&-\left(\frac{\phi-x}{1-x}\right)^2\left(\left(\frac{2}{k}+\left(\frac{k(k-1)}{2}-1\right)\frac{1}{k^2}\right)\log\left(\frac{2}{k}+\left(\frac{k(k-1)}{2}-1\right)\frac{1}{k^2}\right)-\left(\log k\right)\frac{(k-1)(k-2)}{2k^2}\right)\\
&&-\left(\frac{\phi-x}{1-x}\right)^2\sum_{j=0}^{k-1}\frac{1}{k}\left(\frac{j}{k}\right)\log\left(\frac{j}{k}\right)+\left(\frac{\phi-x}{1-x}\right)^2\frac{1}{k}\left(\frac{k-1}{k}\right)\log\left(\frac{k-1}{k}\right)\\
&&+\left(\frac{1-\phi}{1-x}\right)\left(x\log\left(\frac{x}{k}\right)+(1-x)\log(1-x)\right)-\left(\frac{\phi-x}{1-x}\right)\log k-\phi(1-\phi)\log m.
\end{eqnarray*}
This is at least
\begin{eqnarray*}
&&-2\left(\frac{\phi-x}{1-x}\right)\left(\frac{1-\phi}{1-x}\right)\left(\left(\frac{x}{k}+\frac{1}{k}+\left(\frac{k(k-1)}{2}-1\right)\frac{x}{k^2}\right)\log\left(\frac{x}{k}+\frac{1}{k}+\left(\frac{k(k-1)}{2}-1\right)\frac{x}{k^2}\right)+(\log k)\frac{(k-1)(k(x-2)+2x)}{2k^2}\right)\\
&&-2\left(\frac{\phi-x}{1-x}\right)\left(\frac{1-\phi}{1-x}\right)\int_0^1(1-(1-t)x)\log(1-(1-t)x)dt+2\left(\frac{\phi-x}{1-x}\right)\left(\frac{1-\phi}{1-x}\right)\frac{1}{k}\left(1-\frac{x}{k}\right)\log\left(1-\frac{x}{k}\right)\\
&&-\left(\frac{1-\phi}{1-x}\right)^2\left((1-x)^2\log((1-x)^2)+\left(2\frac{x}{k}+\left(\frac{k(k-1)}{2}-1\right)\frac{x^2}{k^2}\right)\log\left(2\frac{x}{k}+\left(\frac{k(k-1)}{2}-1\right)\frac{x^2}{k^2}\right)+\left(\log k\right)\frac{x(k-1)(2x+k(3x-4))}{2k^2}\right)\\
&&-\left(\frac{1-\phi}{1-x}\right)^2\int_0^1\left(2x-\left(2-t\right)x^2\right)\log\left(2-\left(2-t\right)x\right)dt+(\log x)\left(\frac{1-\phi}{1-x}\right)^2\frac{x(x+k(3x-4))}{2k}+\frac{1}{k}\left(\frac{1-\phi}{1-x}\right)^2\left(2x-(k+1)\frac{x^2}{k}\right)\log\left(2x-(k+1)\frac{x^2}{k}\right)\\
&&-\left(\frac{\phi-x}{1-x}\right)^2\left(\left(\frac{2}{k}+\left(\frac{k(k-1)}{2}-1\right)\frac{1}{k^2}\right)\log\left(\frac{2}{k}+\left(\frac{k(k-1)}{2}-1\right)\frac{1}{k^2}\right)-\left(\log k\right)\frac{(k-1)(k-2)}{2k^2}\right)\\
&&-\left(\frac{\phi-x}{1-x}\right)^2\sum_{j=0}^{k-1}\frac{1}{k}\left(\frac{j}{k}\right)\log\left(\frac{j}{k}\right)+\left(\frac{\phi-x}{1-x}\right)^2\frac{1}{k}\left(\frac{k-1}{k}\right)\log\left(\frac{k-1}{k}\right)\\
&&+\left(\frac{1-\phi}{1-x}\right)\left(x\log\left(\frac{x}{k}\right)+(1-x)\log(1-x)\right)-\left(\frac{\phi-x}{1-x}\right)\log k-\phi(1-\phi)\log m\\
&=& -2\left(\frac{\phi-x}{1-x}\right)\left(\frac{1-\phi}{1-x}\right)\left(\left(\frac{x}{k}+\frac{1}{k}+\left(\frac{k(k-1)}{2}-1\right)\frac{x}{k^2}\right)\log\left(\frac{x}{k}+\frac{1}{k}+\left(\frac{k(k-1)}{2}-1\right)\frac{x}{k^2}\right)+(\log k)\frac{(k-1)(k(x-2)+2x)}{2k^2}\right)\\
&&-\frac{1}{2}\left(\frac{\phi-x}{1-x}\right)\left(\frac{1-\phi}{1-x}\right)\left(-2+x-\frac{2(1-x)^2\log(1-x)}{x}\right)+2\left(\frac{\phi-x}{1-x}\right)\left(\frac{1-\phi}{1-x}\right)\frac{1}{k}\left(1-\frac{x}{k}\right)\log\left(1-\frac{x}{k}\right)\\
&&-\left(\frac{1-\phi}{1-x}\right)^2\left((1-x)^2\log((1-x)^2)+\left(2\frac{x}{k}+\left(\frac{k(k-1)}{2}-1\right)\frac{x^2}{k^2}\right)\log\left(2\frac{x}{k}+\left(\frac{k(k-1)}{2}-1\right)\frac{x^2}{k^2}\right)+\left(\log k\right)\frac{x(k-1)(2x+k(3x-4))}{2k^2}\right)\\
&&-\frac{1}{4}\left(\frac{1-\phi}{1-x}\right)^2\left(x(3x-4)-8(1-x)^2\log(2(1-x))+2(2-x)^2\log(2-x)\right)\\
&&+(\log x)\left(\frac{1-\phi}{1-x}\right)^2\frac{x(x+k(3x-4))}{2k}+\frac{1}{k}\left(\frac{1-\phi}{1-x}\right)^2\left(2x-(k+1)\frac{x^2}{k}\right)\log\left(2x-(k+1)\frac{x^2}{k}\right)\\
&&-\left(\frac{\phi-x}{1-x}\right)^2\left(\left(\frac{2}{k}+\left(\frac{k(k-1)}{2}-1\right)\frac{1}{k^2}\right)\log\left(\frac{2}{k}+\left(\frac{k(k-1)}{2}-1\right)\frac{1}{k^2}\right)-\left(\log k\right)\frac{(k-1)(k-2)}{2k^2}\right)\\
&&-\left(\frac{\phi-x}{1-x}\right)^2\sum_{j=0}^{k-1}\frac{1}{k}\left(\frac{j}{k}\right)\log\left(\frac{j}{k}\right)+\left(\frac{\phi-x}{1-x}\right)^2\frac{1}{k}\left(\frac{k-1}{k}\right)\log\left(\frac{k-1}{k}\right)\\
&&+\left(\frac{1-\phi}{1-x}\right)\left(x\log\left(\frac{x}{k}\right)+(1-x)\log(1-x)\right)-\left(\frac{\phi-x}{1-x}\right)\log k-\phi(1-\phi)\log m
\end{eqnarray*}
\end{small}

\begin{proposition}For integers $k\geq 5$ and $1\leq m\leq \sqrt{k}$, and for every $0<\phi<\frac{1}{2}$ and $x\in [0,\phi]$, we have
\[F_{k,m}\left(\left(\frac{\phi-x}{1-x}\right)\delta_1+\left(\frac{1-\phi}{1-x}\right)\delta_x\right)>0.\]
\end{proposition}
\begin{proof}
We first show the inequality for $k\geq 13$. Let
\begin{small}
\begin{eqnarray*}
&&B(k;x,\phi):=-2\left(\frac{\phi-x}{1-x}\right)\left(\frac{1-\phi}{1-x}\right)\left(\left(\frac{x}{k}+\frac{1}{k}+\left(\frac{k(k-1)}{2}-1\right)\frac{x}{k^2}\right)\log\left(\frac{x}{k}+\frac{1}{k}+\left(\frac{k(k-1)}{2}-1\right)\frac{x}{k^2}\right)+(\log k)\frac{(k-1)(k(x-2)+2x)}{2k^2}\right)\\
&&-\frac{1}{2}\left(\frac{\phi-x}{1-x}\right)\left(\frac{1-\phi}{1-x}\right)\left(-2+x-\frac{2(1-x)^2\log(1-x)}{x}\right)+2\left(\frac{\phi-x}{1-x}\right)\left(\frac{1-\phi}{1-x}\right)\frac{1}{k}\left(1-\frac{x}{k}\right)\log\left(1-\frac{x}{k}\right)\\
&&-\left(\frac{1-\phi}{1-x}\right)^2\left((1-x)^2\log((1-x)^2)+\left(2\frac{x}{k}+\left(\frac{k(k-1)}{2}-1\right)\frac{x^2}{k^2}\right)\log\left(2\frac{x}{k}+\left(\frac{k(k-1)}{2}-1\right)\frac{x^2}{k^2}\right)+\left(\log k\right)\frac{x(k-1)(2x+k(3x-4))}{2k^2}\right)\\
&&-\frac{1}{4}\left(\frac{1-\phi}{1-x}\right)^2\left(x(3x-4)-8(1-x)^2\log(2(1-x))+2(2-x)^2\log(2-x)\right)\\
&&+(\log x)\left(\frac{1-\phi}{1-x}\right)^2\frac{x(x+k(3x-4))}{2k}+\frac{1}{k}\left(\frac{1-\phi}{1-x}\right)^2\left(2x-(k+1)\frac{x^2}{k}\right)\log\left(2x-(k+1)\frac{x^2}{k}\right)\\
&&-\left(\frac{\phi-x}{1-x}\right)^2\left(\left(\frac{2}{k}+\left(\frac{k(k-1)}{2}-1\right)\frac{1}{k^2}\right)\log\left(\frac{2}{k}+\left(\frac{k(k-1)}{2}-1\right)\frac{1}{k^2}\right)-\left(\log k\right)\frac{(k-1)(k-2)}{2k^2}\right)\\
&&+\left(\frac{\phi-x}{1-x}\right)^2\frac{1}{k}\left(\frac{k-1}{k}\right)\log\left(\frac{k-1}{k}\right)+\left(\frac{1-\phi}{1-x}\right)\left(x\log\left(\frac{x}{k}\right)+(1-x)\log(1-x)\right)-\left(\frac{\phi-x}{1-x}\right)\log k-\frac{1}{2}\phi(1-\phi)\log k.
\end{eqnarray*}
\end{small}

Note that
\begin{small}
\[F_{k,m}\left(\left(\frac{\phi-x}{1-x}\right)\delta_1+\left(\frac{1-\phi}{1-x}\right)\delta_x\right)\geq B(k;x,\phi).\]
\end{small}
We calculate
\begin{small}
\begin{eqnarray*}&&k^3\frac{\partial}{\partial k}B(k;x,\phi)\\
&=&-2\left(\frac{\phi-x}{1-x}\right)\left(\frac{1-\phi}{1-x}\right)\left(\frac{4x-k(x+2)}{2}\left(1+\log\left(\frac{x}{k}+\frac{1}{k}+\left(\frac{k(k-1)}{2}-1\right)\frac{x}{k^2}\right)\right)+\frac{(k-1)(k(x-2)+2x)-(k(x+2)-4x)\log k}{2}\right)\\
&&+2\left(\frac{\phi-x}{1-x}\right)\left(\frac{1-\phi}{1-x}\right)\left(x-(k-2x)\log\left(1-\frac{x}{k}\right)\right)\\
&&-\left(\frac{1-\phi}{1-x}\right)^2\left(\frac{x(k(x-4)+4x)}{2}\left(1+\log\left(2\frac{x}{k}+\left(\frac{k(k-1)}{2}-1\right)\frac{x^2}{k^2}\right)\right)+\frac{x((k-1)(k(3x-4)+2x)+(k(x-4)+4x)\log k)}{2}\right)\\
&&-\frac{kx^2\log x}{2}\left(\frac{1-\phi}{1-x}\right)^2+x\left(\frac{1-\phi}{1-x}\right)^2\left(x+(k(x-2)+2x)\log\left(\frac{x(k(2-x)-x)}{k}\right)\right)\\
&&-\left(\frac{\phi-x}{1-x}\right)^2\left(\frac{2-k^2+(4-3k)\log\left(\frac{3-\frac{2}{k}+k}{2}\right)}{2}\right)+\left(\frac{\phi-x}{1-x}\right)^2\left(1-(k-2)\log\left(\frac{k-1}{k}\right)\right)-k^2\phi(1+\frac{1}{2}(1-\phi))
\end{eqnarray*}
\end{small}

Furthermore, we obtain

\begin{small}
\begin{eqnarray*}
&&\frac{\partial}{\partial k}\left(k^3\frac{\partial}{\partial k}B(k;x,\phi)\right)\\
&=&-2\left(\frac{\phi-x}{1-x}\right)\left(\frac{1-\phi}{1-x}\right)\left(\frac{1}{2}\left(\frac{(k(x+2)-4x)^2}{k(k^2x+k(x+2)-2x)}-(x+2)\left(1+\log\left(\frac{k^2x+k(x+2)-2x}{2k^2}\right)\right)\right)+\frac{2x}{k}-\frac{1}{2}(x+2)\log k\right)\\
&&-2\left(\frac{\phi-x}{1-x}\right)\left(\frac{1-\phi}{1-x}\right)\left(\frac{x(k-2x)}{k(k-x)}+\log\left(1-\frac{x}{k}\right)\right)\\
&&-\left(\frac{1-\phi}{1-x}\right)^2\left(\frac{x}{2}\left(\frac{(k(x-4)+4x)^2}{k(k^2x-k(x-4)-2x)}+(x-4)\left(1+\log\left(\frac{x(k^2x-k(x-4)-2x)}{2k^2}\right)\right)\right)+\frac{x(k(x-4)(\log k)+4x)}{2k}\right)\\
&&-\frac{x^2\log x}{2}\left(\frac{1-\phi}{1-x}\right)^2+x\left(\frac{1-\phi}{1-x}\right)^2\left((x-2)\log\left(\frac{x((2-x)k-x)}{k}\right)+\frac{x((2-x)k-2x)}{k(k(x-2)+x)}\right)\\
&&-\frac{1}{2}\left(\frac{\phi-x}{1-x}\right)^2\left(\frac{\left(\frac{2}{k^2}+1\right)(4-3k)}{k+3-\frac{2}{k}}-3\log\left(\frac{1}{2}\left(k+3-\frac{2}{k}\right)\right)\right)-\left(\frac{\phi-x}{1-x}\right)^2\left(\frac{k-2}{k(k-1)}+\log\left(\frac{k-1}{k}\right)\right)+k\phi(1-2\phi)\\
&\geq& -2\left(\frac{\phi-x}{1-x}\right)\left(\frac{1-\phi}{1-x}\right)\left(-\frac{x}{2}+\frac{3x}{k}\right)-\left(\frac{1-\phi}{1-x}\right)^2\left(\frac{x}{2}\left(x-(4-x)\log\left(\frac{x(4-x)}{2}\right)\right)+\frac{3x^2}{k}\right)\\
&&-\frac{x^2\log x}{2}\left(\frac{1-\phi}{1-x}\right)^2+x\left(\frac{1-\phi}{1-x}\right)^2\left((x-2)\log\left(\frac{x((2-x)k-x)}{k}\right)\right)+\frac{3}{2}\left(\frac{\phi-x}{1-x}\right)^2\log\left(\frac{k}{2}\right)+k\phi(1-2\phi)\\
&\geq& -\left(\frac{1-\phi}{1-x}\right)^2\left(\frac{x}{2}\left(x-(4-x)\log\left(\frac{x(4-x)}{2}\right)\right)+\frac{3x^2}{k}+\frac{x^2\log x}{2}-x\left((x-2)\log\left(\frac{x((2-x)k-x)}{k}\right)\right)\right),
\end{eqnarray*}
\end{small}
where the last inequality is true if $k\geq 6$. It is easy to see that this is non-negative for $k\geq 6$. Furthermore, it is easy to show that
\begin{small}
\[\frac{\partial}{\partial k}\Bigg|_{k=13}B(k;x,\phi)>0.\]
\end{small}
To show this, we note that it is a strictly concave function of $0<\phi<\frac{1}{2}$, and so we may verify non-negativity at the endpoints, which is easily done. Therefore, to get the result for $k\geq 13$, it suffices to verify that
\begin{small}
\[B(13;x,\phi)>0\]
\end{small}
which may again be verified by noting that it is strictly concave in $0<\phi<\frac{1}{2}$ and reducing to checking non-negativity at the boundaries. We are reduced to verifying that for $5\leq k\leq 12$, we have the right inequality in
\begin{small}
\[F_{k,m}\left(\left(\frac{\phi-x}{1-x}\right)\delta_1+\left(\frac{1-\phi}{1-x}\right)\delta_x\right)\geq F_{k,\lfloor\sqrt{k}\rfloor}\left(\left(\frac{\phi-x}{1-x}\right)\delta_1+\left(\frac{1-\phi}{1-x}\right)\delta_x\right)>0\]
\end{small}
These may be verified using similar methods.
\end{proof}

\end{document}